\documentclass[a4paper]{amsart}

\usepackage{cmll}
\usepackage{enumerate}
\usepackage{mathtools}
\usepackage{hyperref}
\usepackage{amssymb}

\title{The clumsy coupon collector's problem}

\author{Luke J. Attrill}
\address{School of Mathematics, Monash University VIC 3800, Australia}
\thanks{We thank Stephen Muirhead for many fruitful discussions, particularly concerning the critical case.
We also thank Christopher D.~Long for sharing with us a preliminary draft of his article~\cite{long2026clumsy}.}
\email{Luke.Attrill@monash.edu}
\thanks{Supported by an Australian Government Research
and Training Program Scholarship
[DOI: \href{https://doi.org/10.82133/C42F-K220}{\texttt{10.82133/C42F-K220}}].}
\author{Timothy M. Garoni}
\address{School of Mathematics, Monash University VIC 3800, Australia}
\email{Tim.Garoni@monash.edu}
\thanks{Supported by the Australian Research Council's \emph{Discovery Projects}
funding scheme (project DP230102209).}

\date{24 November 2026}

\theoremstyle{plain}
  \newtheorem{theorem}{Theorem}[section]
  \newtheorem{proposition}[theorem]{Proposition}
  \newtheorem{lemma}[theorem]{Lemma}
  \newtheorem{corollary}[theorem]{Corollary}
\theoremstyle{definition}
  \newtheorem{definition}[theorem]{Definition}

\numberwithin{equation}{section}

\let\originalmiddle=\middle
\def\middle#1{\mathrel{}\originalmiddle#1\mathrel{}}

\let\leq=\leqslant
\let\le=\leqslant
\let\geq=\geqslant
\let\ge=\geqslant
\newcommand{\C}{\mathsf{c}} 
\newcommand{\D}{\mathsf{d}} 

\newcommand{\stirling}{\genfrac{[}{]}{0pt}{}}
\DeclareMathOperator{\Var}{Var}
\DeclareMathOperator{\E}{E}
\DeclareMathOperator{\PP}{P}
\DeclareMathOperator{\Gammaop}{\Gamma}
\DeclareMathOperator{\Betaop}{B}

\newcommand{\NN}{\mathbb{N}}
\newcommand{\CC}{\mathbb{C}}
\newcommand{\RR}{\mathbb{R}}
\newcommand{\1}{1} 
\newcommand{\diff}{\mathrm{d}}
\newcommand{\weak}{\stackrel{\rm w}{\to}} 
\newcommand{\prob}{\stackrel{\PP}{\to}} 
\newcommand{\sA}{\mathcal{A}}
\newcommand{\sC}{\mathcal{C}}
\newcommand{\sD}{\mathcal{D}}
\newcommand{\sE}{\mathcal{E}}
\newcommand{\sL}{\mathcal{L}}
\newcommand{\sS}{\mathcal{S}}

\begin{document}


\begin{abstract}
We consider a generalisation of the classical coupon
collector's problem, in which at each time step a collector either
receives a new copy of a randomly chosen coupon, or loses all their
previously collected copies of that coupon. We consider the amount
of time it takes this \emph{clumsy} coupon collector to obtain the
full set of $m$ coupons. We establish limit theorems as $m\to\infty$
for the clumsy coupon collection time, and describe the large $m$
asymptotics of its mean and variance. We identify three regimes,
depending on how the probability of a clumsy update, $p$, 
scales with $m$. If $p=o(1/m)$, we obtain a Gumbel limit theorem,
as is the case for the classical coupon collector. If $p=\omega(1/m)$,
we instead show weak convergence to an exponential random variable.
In the critical case, $p\sim c/m$, we give a full characterisation
of the limiting distribution in terms of a birth--death process.
We also describe the asymptotics of the mean and variance in these three regimes.
\end{abstract}

\maketitle

\section{Introduction}
\label{introduction}

We study the clumsy coupon collector's problem, introduced in~\cite{garoni2025matrix}.
Consider $m$ urns, labelled $1,2,\ldots,m$, each corresponding to a coupon type. The urns are initially empty.
Each day, a collector is issued a coupon, chosen uniformly at random from ${1,2,\ldots,m}$.
With probability $1-p$, the collector adds the coupon to the corresponding urn.
Otherwise, with probability $p$, the collector clumsily knocks over the urn
and loses all coupons of that type. The first day, $T_{m,p}$,
on which every urn is non-empty is called the \emph{clumsy coupon collection time}.
See Figure~\ref{fig:example} for an illustration. A notable feature of the clumsy
coupon collector's problem is that, unlike in the classical version,
the collector's progress is non-monotone. Consequently,
$T_{m,p}$ is subject to competing effects and exhibits a phase transition as
$p$ is varied across the critical scale $p\asymp 1/m$.

We prove that, appropriately standardised, $T_{m,p}$
has a Gumbel limit when $p=o(1/m)$, an exponential limit when
$p=\omega(1/m)$, and a limiting distribution characterised in terms of
a birth--death process in the intermediate regime $p\sim c/m$, with
$c>0$ fixed.
We also provide detailed asymptotic results for the mean
$\E T_{m,p}$ and variance $\Var(T_{m,p})$ in these three regimes.

The clumsy coupon collection time can be defined formally as follows. Throughout,
$\NN\coloneq\{1,2,\dots\}$ denotes the positive integers, and
for $m\in\NN$ we set $[m]\coloneq\{1,2,\dots,m\}$.
\begin{definition}
\label{def:coupling}
Let $\{U_n,C_n\}_{n\in\mathbb N}$ be a family of independent random
variables such that, for each $n\in\mathbb N$, $C_n$ is uniformly
distributed on $[m]$ and $U_n$ is uniformly distributed on $(0,1)$.
For each $i\in[m]$ and $n\in\NN$, define
\[
  L(i,n)\coloneq\sup\{k\leq n\mid C_k=i\}.
\]
Thus, $L(i,n)$ is the most recent day no later than $n$ on which
coupon type $i$ was issued, if such a day exists; otherwise, $L(i,n)=-\infty$.
The clumsy coupon collection time with $m$ coupon types and clumsiness
probability $p\in[0,1)$ is 
\[
  T_{m,p}\coloneq\inf\{n>0
    \mid L(i,n)>0\ \text{and}\ U_{L(i,n)}< 1-p
    \ \text{for all $i\in[m]$}\}.
\]
For readability, when $T_{m,p}$ appears as a subscript, we instead
write $T(m,p)$.
\end{definition}

\begin{figure}[t]
  \label{fig:example}
  \centering
\includegraphics{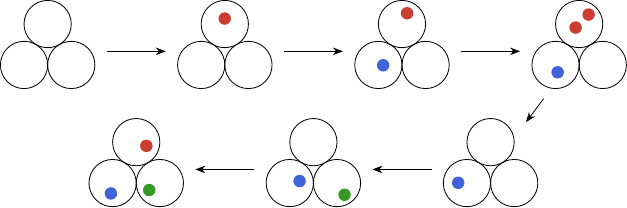}
  \caption{One particular realisation in the event $\{ T_{3, p} = 6 \}$, where $p > 0$.
  Red, green and blue denote coupon types $1$, $2$ and $3$ respectively. On day $4$, the red urn is clumsily knocked over.}
\end{figure} 

We emphasise that for each given $m\in\NN$, Definition~\ref{def:coupling}
defines a coupling of $T_{m,p}$ for all values of $p\in[0,1)$. 
In particular, this allows direct comparison between the classical coupon
collection time, $T_{m,0}$, and the clumsy coupon collection time $T_{m,p}$
with $p>0$. Clearly $T_{m, 0}\leq T_{m,p}$, and so $T_{m, p} - T_{m, 0}$
is a non-negative random variable defined on the same space as $T_{m,0}$. 
Moreover, we show in Lemma~\ref{thm:coupling} that $T_{m, 0}$ and
$T_{m, p} - T_{m, 0}$ are independent:
\[
  T_{m, 0} \Perp T_{m,p} - T_{m,0}.
\]
As we will see, it is the competition between $T_{m,0}$ and $T_{m,p}-T_{m,0}$
that causes the distinct behaviours of $T_{m,p}$. 

In what follows, we will in general assume that $p=p(m)\in[0,1)$ depends on $m$.
For such a sequence, we apply the coupling of Definition~\ref{def:coupling}  
separately for each $m$, thereby coupling $T_{m,p(m)}$ with $T_{m,0}$.
In the interests of readability, we will henceforth suppress the
dependence on $m$ and simply write $p$ for $p(m)$.

One indication that $p \asymp 1/m$ should be the critical scale
follows from the elementary observation that for $c\in(0,\infty)$,
\begin{equation}
  \label{eq:prob that clumsy eq classical}
  \PP(T_{m,0} = T_{m,p}) = (1-p)^m \xrightarrow[m\to\infty]{}
  \begin{cases}
    1, & p=o(1/m),\\
    e^{-c}, & p\sim c/m,\\
    0, & p=\omega(1/m).
  \end{cases}
\end{equation}
Combining~\eqref{eq:prob that clumsy eq classical} with the classical
result~\cite{erdos1961coupon} that $(T_{m,0}- m\log m)/m$ converges
in distribution to a standard Gumbel random variable immediately yields the following.
\begin{theorem}[Subcritical shape]
  \label{thm:subcritical shape}
  If $p = o(1/m)$, then $(T_{m,p} - m \log{m})/m$ converges
  in distribution to a standard Gumbel random variable.
\end{theorem}

By contrast, suppose $p = \omega(1/m)$. In this case we have a qualitatively different behaviour.
\begin{theorem}[Supercritical shape]
  \label{thm:supercritical shape}
  If $p = \omega(1 / m)$, then $p \, (1-p)^m \,
  (T_{m,p} - m \log{m})$ converges in law to an exponential distribution with rate $1$.
\end{theorem}
Given that $T_{m,p} > T_{m,0}$ asymptotically almost surely when $p=\omega(1/m)$,
one might question the importance of shifting  $T_{m,p}$ by $m\log(m) \sim \E T_{m,0}$.
However, the subtraction of $m\log m$ is necessary in general because the assumption
$p=\omega(1/m)$ implies only that $p^{-1}(1-p)^{-m}=\omega(m)$.
Thus, while $p^{-1}(1-p)^{-m}$ dominates the $O(m)$ standard deviation of $T_{m,0}$,
it need not dominate its mean. For example, $p = (\log\log{m})/m$
implies $p^{-1}(1-p)^{-m} \sim m \log m/\log\log{m}=o(m\log m)$.

The critical case $p\sim c/m$ is perhaps the most interesting.
First observe that if $N_{m,p}$ denotes the number of empty urns at
time $T_{m,0}$, then $N_{m,p}\sim\operatorname{Bin}(m,p)$. 
As $m\to\infty$, it follows that $N_{m,p}\to0$ in probability
for $p=o(1/m)$, but $N_{m,p}\to\infty$ in probability for $p=\omega(1/m)$.
It has a non-degenerate limit, with distribution $\operatorname{Poi}(c)$,
precisely when $p\sim c/m$. From a state with $r$ empty urns, the probability
of moving to $r-1$ empty urns at the next step is $(1-p)\,r/m\sim r/m$,
whereas the probability of moving to $r+1$ empty urns is $p\,(m-r)/m \sim c/m$.
Since changing the number of empty urns requires of order $m$ updates, 
this suggests that $(T_{m,p}-T_{m,0})/m$ converges to the hitting time
of $0$ for the following continuous-time birth--death process. 
\begin{definition}
  \label{def:birth-death}
  Fix $c > 0$ and let $(Q_t)_{t\geq 0}$
  be the birth--death process with
  \begin{enumerate}[(i)]
    \item $Q_0 \sim \operatorname{Poisson}(c)$;
    \item $n \mapsto n +1$ transitions of rate $c$; and
    \item $n \mapsto n - 1$ transitions of rate $n$.
  \end{enumerate}
  Let $\tau_c \coloneq \inf\{ t \geq 0 \mid Q_t = 0 \}$ denote the first hitting time of $0$.
\end{definition}
Combining the above heuristics for $(T_{m,p}-T_{m,0})/m$ with
the known behaviour of $(T_{m,0}-m\log m)/m$ then
suggests that $(T_{m,p}-m\log m)/m$ should behave like the sum
of $\tau_{c}$ and a standard Gumbel random variable.
As the next theorem states, this is precisely what occurs. 
\begin{theorem}[Critical shape]
  \label{thm:critical shape}
Fix $c > 0$ and let $p \sim c/m$. Let $\tau_c$ be as defined in
Definition~\ref{def:birth-death}, and let $X$ be a standard
Gumbel random variable, independent of $\tau_c$.
Then $(T_{m,p} - m\log{m})/m$ converges in law to $X + \tau_c$.
\end{theorem}
 
We prove the above limit theorems by showing pointwise convergence
of the associated moment-generating functions.
An explicit form for the probability-generating function of $T_{m,p}$,
valid for all $m$ and $p$, is presented in Theorem~\ref{thm:pgf}.
The proof of Theorem~\ref{thm:pgf} utilises the machinery of
language-generating functions~\cite{flajolet1992birthday}. Theorem~\ref{thm:pgf}
then allows us to derive an integral expression for the
moment-generating function, given in Lemma~\ref{thm:factored mgf},
which is well-suited for asymptotic analysis.

In addition to establishing the above limit theorems, our expressions
for the moment-generating function allow us to derive detailed
asymptotics for the mean and variance. The coupling in
Definition~\ref{def:coupling} yields the decompositions
\begin{align*}
\E T_{m,p} &= \E T_{m,0} + \E\left[T_{m,p}-T_{m,0}\right],\\
\Var\left(T_{m,p}\right)
&= \Var\left(T_{m,0}\right)
+ \Var\left(T_{m,p}-T_{m,0}\right).
\end{align*}
Moreover, we recall that the mean and variance of $T_{m,0}$
have the classical representations~\cite[Equations (12), (13)]{schelling1954coupon} 
\[
  \E T_{m,0}=m\,H_m,
  \qquad
  \Var\left(T_{m,0}\right)
  =m^2\,H_m^{(2)}-m\,H_m,
\]
where $H_m$ denotes the $m$th harmonic number, and $H_m^{(2)}$ denotes
the $m$th harmonic number of order $2$.
The well-known complete asymptotic expansions~\cite{wilf1993} of $H_m$
and $H_m^{(2)}$ thus immediately yield corresponding expansions for $\E T_{m,0}$
and  $\Var\left(T_{m,0}\right)$. It therefore suffices to consider
the mean and variance of the difference $T_{m,p}-T_{m,0}$.

\begin{theorem}[Mean asymptotics]
  \label{thm:first moment}
  As $m\to\infty$  
  \[
    \E \left[ T_{m, p} - T_{m, 0} \right]
    = \begin{cases}
      p\,m^2 + \frac{1}{4} \, p^2 \, m^3
      + O\left(p^2 \, m^2\right) + O\left(p^3 \, m^4\right),
      & p = o(1/m),\\
      m \, A(c)
      \left(1 + o(1)\right),
      & p \sim c/m,\\ 
      p^{-1}\left(1-p\right)^{-m} \left(1 + O\left(\frac{1}{p\,m}\right)\right),
      & p = \omega(1/m),
    \end{cases}
  \]
  for any fixed $c\in(0,\infty)$.
  The function $A$ is strictly increasing on $(0, \infty)$, and
  satisfies $A(c) = c + O(c^2)$ as $c \downarrow 0$,
  and $A(c) = (e^c / c)\,(1 + O(1/c))$ as $c \to \infty$.

  If $p\in(0,1)$ is fixed, then as $m \to \infty$
  \begin{align*}
    \E \left[T_{m,p} - T_{m,0}\right]
    = \frac{(1-p)^{-m}}{p}\left\{
      1 + \frac{p^{-1} - 1}{m}
      + \frac{1 - 3\,p^{-1} + 2\,p^{-2}}{m^2}
      + O\left(m^{-3}\right)
    \right\}.
  \end{align*}
  The same asymptotic expansion holds with $\E[T_{m,p}-T_{m,0}]$
  replaced by $\E T_{m,p}$.
\end{theorem}

An explicit expression for $A(c)$ is given in \eqref{eq:A}.

\begin{theorem}[Variance asymptotics]
  \label{thm:second moment}
  As $m\to\infty$
  \[
    \Var\left(T_{m,p} - T_{m,0} \right)
    =
    \begin{cases}
    2\,p\,m^3 - p\,m^2
    + \frac{5}{4} \, p^2\, m^4
    + O\left(p^3\,m^5\right)
    + O\left(p^2 \, m^3\right),
    & p = o(1/m),\\
    m^2 \left( A(c)^2 + B(c)\right) \left(1 + o(1)\right),
    & p \sim c/m,\\
    p^{-2}\left(1-p\right)^{-2m}
    \left( 1 + O\left(\frac{1}{p\,m} \right)\right),
    & p = \omega(1/m),
    \end{cases}
  \]
  for any fixed $c\in(0,\infty)$.
  The function $A$ is as described in Theorem~\ref{thm:first moment}.
  The function $B$ is strictly increasing on $(0, \infty)$, and
  satisfies $B(c) = 2\,c + O(c^2)$ as $c \downarrow 0$,
  and $B(c) = (2\,e^c\,c^{-2})\,(1 + O(1/c))$ as $c \to \infty$.

  If $p\in(0,1)$ is fixed, then as $m\to \infty$
  \[
    \Var\left(T_{m,p} - T_{m,0} \right)
    = \frac{(1-p)^{-2m}}{p^2}
    \left(1 + \frac{2\,p^{-1} - 2}{m} +
    \frac{5\,p^{-2} - 8\,p^{-1} + 3}{m^2} + O\left(m^{-3}\right)\right).
  \]
The same asymptotic expansion holds with $\operatorname{Var}(T_{m,p}-T_{m,0})$
replaced by $\operatorname{Var}(T_{m,p})$.
\end{theorem}
An explicit expression for $B(c)$ is given in~\eqref{eq:B}.

\subsection*{History and related work}
The study of the classical coupon collector’s problem dates back at
least to 1711, when de Moivre~\cite{demoivre1712demensura,deMoivre1984} 
obtained an explicit expression for $\PP(T_{m,0}\leq n)$.
Apparently, the problem was posed to de Moivre by the English
politician Francis Robartes, although it was originally phrased
as a betting game between two players rolling an $m$-sided
die~\cite[Art.~236]{todhunter1865history}. Schelling began
the modern study of $T_{m,0}$ by describing its mean and
variance~\cite{schelling1934coupon}. Independently,
Maunsell~\cite{maunsell1938coupon} and Watson~\cite{watson1939coupon}
showed $\E T_{m, 0} = m \, H_m$. Godwin~\cite{godwin1949coupon}
subsequently obtained all the moments of $T_{m,0}$. 
It seems that Feller cemented the name ``coupon collector'' in
his textbook~\cite[pp.~174--175]{feller1950probability},
where he gave a clear treatment of the topic. Erd\H{o}s and
R\'enyi~\cite{erdos1961coupon} proved the Gumbel limit of
$(T_{m, 0}-m\log m)/m$. Billingsley and Baum later found the
asymptotic shapes which result from a partial collection of the
coupons~\cite{billingsley1965coupon}. Newman and Shepp~\cite{newman1960dixie}
studied the time required for $k$ copies of each coupon to be collected.

A natural generalisation of the classical case is to allow the
distribution of coupon types to be non-uniform.
This was already considered by Schelling~\cite{schelling1934coupon, schelling1954coupon}.
Boneh and Hofri~\cite{boneh1989revisited}, and later
Doumas and Papanicolaou~\cite{doumas2016dixie}, studied the time
required to collect $k$ copies of each coupon type in the non-uniform setting.

Contemporary research has spawned a myriad of variants.
Myers and Wilf consider two simultaneous collectors~\cite{wilf2006coupon}.
Jockovi\'c and Todi\'c introduce a global reset coupon~\cite{jockovic2024reset}.
Recently, Cruciani and Dudeja proposed a careless variant in which each urn is
independently emptied each day, regardless of which
coupon was offered~\cite{cruciani2026careless}.

During the final stages of preparation of the current
paper we were made aware of Long's preprint~\cite{long2026clumsy},
which considers the special case of the clumsy coupon collector
problem in which $p>0$ is fixed, independent of $m$. 
The fixed $p>0$ case of Theorem~\ref{thm:supercritical shape} coincides
with Long's Theorem 4.4.
And the leading-order terms of the mean and variance
asymptotics in the fixed $p>0$ case given in
Theorems~\ref{thm:first moment} and~\ref{thm:second moment} appear in~\cite[Corollary 4.5]{long2026clumsy}.
Our methods of proof are entirely different, however.

\subsection*{Outline}
The current paper is organised as follows. 
Section~\ref{language} reviews language-generating functions,
the central technique utilised in this paper.
Section~\ref{first_collection} then derives the
probability-generating function of $T_{m,p}$ using language-generating functions.
From this, we give exact expressions for the expectation $\E T_{m,p}$ and variance $\Var(T_{m, p})$.
Section~\ref{asymptotics} focuses
on limit theorems, using the moment-generating function to
prove Theorems~\ref{thm:subcritical shape},~\ref{thm:supercritical shape}
and~\ref{thm:critical shape}. Section~\ref{practical} then studies
the asymptotics of $\E T_{m,p}$ and $\Var(T_{m,p})$, establishing
Theorem~\ref{thm:first moment} and Theorem~\ref{thm:second moment}.

\section{Language-generating functions}
\label{language}
This section recalls results on language-generating functions from~\cite{flajolet1992birthday},
which we make use of in Section~\ref{first_collection}.
Let $A$ be a finite set called \emph{the alphabet} whose elements
are called \emph{letters}.
Let $W$ consist of all \emph{words} (finite lists of letters)
with $\varepsilon$ denoting the empty word.
If we define $A^0 \coloneq \{ \varepsilon \}$,
then $W \coloneq \bigcup_{n \geq 0} A^n$.
Any subset
$L \subseteq W$ is called a \emph{language}.
The \emph{concatenation} of two languages $L_1$ and $L_2$ is
the collection of all concatenations $w_1.w_2$ where $w_1 \in L_1$
and $w_2 \in L_2$,
\[
  L_1 . L_2 \coloneq \bigcup_{\substack{w_1 \in L_1\\ w_2 \in L_2}}
  \{w_1 .w_2\}.
\]
The \emph{shuffling} of two words is a set of words, defined recursively by:
\begin{align*}
  w \circ \varepsilon = \varepsilon \circ w &\coloneq \{ w \},
  && w \in W,\\
{(a_1.w_1)} \circ {(a_2.w_2)}
        &\coloneq {\{a_1\}.(w_1 \circ a_2.w_2)}
          \cup {\{a_2\}.(a_1.w_1 \circ w_2)}, && a_1,a_2 \in A,\ w_1,w_2 \in W.
\end{align*}
For example, the shuffling of $\mathsf{be}$ and $\mathsf{dog}$ is
the set
\[
  \mathsf{be} \circ \mathsf{dog} = \{ \mathsf{bedog},
  \mathsf{bdeog}, \mathsf{bdoeg}, \mathsf{bdoge},
  \mathsf{dbeog}, \mathsf{dboeg}, \mathsf{dboge},
  \mathsf{dobeg}, \mathsf{dobge}, \mathsf{dogbe} \}.
\]
Necessarily, the letters in $\mathsf{be}$ and $\mathsf{dog}$ preserve their order.
The shuffling of two languages is
\[
  L_1 \circ L_2 \coloneq \bigcup_{\substack{w_1 \in L_1\\ w_2 \in L_2}} w_1 \circ w_2,
  \qquad L_1, L_2 \subseteq W.
\]
Concatenation is associative. Shuffling is both commutative and associative.

Assign each letter $a \in A$ a weight $\pi(a) \in [0,1]$ so that $\sum_{a \in A} \pi(a) = 1$.
This can be extended to words by defining $\pi(a_1 a_2 \ldots a_k) \coloneq \prod_{i \in [k]} \pi(a_i)$
with $\pi(\varepsilon) = 1$. Hence, for any fixed $n$,
the restriction of $\pi$ to $A^n$ is a probability mass function.

Define the \emph{ordinary generating function} of $L \subseteq W$ with respect to $\pi$ by
\[
  \phi_L(z) \coloneq \sum_{w \in L} \pi(w) \, z^{\lvert w \rvert},
  \quad z \in \mathbb{D},
\]
where $\mathbb{D}$ denotes the open unit disc $\mathbb{D}
\coloneq \{ z \in \CC : \lvert z \rvert < 1 \}$.
The \emph{exponential generating function} of $L$ is 
\[
  \hat\phi_L(z) \coloneq
  \sum_{w \in L} \pi(w) \, \frac{z^{\lvert w \rvert}}{\lvert w \rvert!},
  \quad z \in \CC.
\]
The coefficient $[z^n]\,\phi_L(z)$ is the probability that a word randomly sampled from $\pi|_{A^n}$ belongs to $L$. 
The following propositions can be found in~\cite{flajolet1992birthday,sedgewick2009combinatorics}.
\begin{proposition}
\label{thm:concatenation}
Suppose words in $L_1.L_2$ have a \emph{unique factorisation}
in the sense that, whenever $w_1.w_2 = v_1.v_2$
for $w_1,v_1 \in L_1$ and $w_2,v_2 \in L_2$,
then $w_1 = v_1$ and $w_2 = v_2$. Then
\[
  \phi_{L_1.L_2} = {\phi_{L_1}} \cdot {\phi_{L_2}}.
\]
\end{proposition}
Any pair of languages that satisfy the assumption in
Proposition~\ref{thm:concatenation} is said to be \emph{well-defined for concatenation}.
\begin{proposition}
\label{thm:shuffling}
If $L_1$ uses disjoint letters from $L_2$ then
\[
  \hat\phi_{L_1\circ L_2} = \hat\phi_{L_1} \cdot \hat\phi_{L_2}.
\]
\end{proposition}
Any pair of languages that satisfy the assumption in
Proposition~\ref{thm:shuffling} is said to be \emph{well-defined for shuffling}.
\begin{proposition}
  \label{thm:union}
If $L_1$ and $L_2$ are disjoint languages then
\[
  \phi_{L_1 \cup L_2} = \phi_{L_1} + \phi_{L_2}
  \quad \text{and} \quad
  \hat\phi_{L_1 \cup L_2} = \hat\phi_{L_1} + \hat\phi_{L_2}
\]
\end{proposition}
Similarly, we call two languages \emph{well-defined for addition} when they meet the assumption of Proposition~\ref{thm:union}.

Finally, we recall the following result, relating ordinary and exponential generating functions.
\begin{proposition}[Laplace--Borel transform]
  \label{thm:LB transform}
  The ordinary and exponential generating functions of $L \subseteq W$ satisfy
  \[
    \phi_L(x) = \int_0^\infty \hat\phi_L(xt) \,e^{-t} \, \diff t,
    \quad x \in (-1, 1).
  \]
\end{proposition}
The existence of the above integral follows from the fact that $\lvert\hat\phi_L(xt)\rvert \leq e^{|x|t}$. Since $\phi_L$ is analytic,
Proposition~\ref{thm:LB transform} fully specifies $\phi_L$ via the restriction of $\hat\phi_L$ to $\RR$. 

\section{The clumsy coupon collection time $T_{m,p}$}
\label{first_collection}
\subsection{The probability-generating function}
 
We now apply the machinery of language-generating functions to study
the clumsy coupon collection time. The main result of this section
is the following characterisation of the probability generating function of $T_{m,p}$.
\begin{theorem}
  \label{thm:pgf}
  For all $p\in[0,1)$ and $m\in\NN$, there exists $\delta_{m, p} > 0$
  such that the probability-generating function of $T_{m,p}$ satisfies
  \[
    G_{T(m,p)}(z)
    \coloneq \sum_{n=0}^\infty \PP(T_{m,p} = n)\, z^n
    = \left[
        \sum_{\ell=0}^m
        (-1)^\ell
        \binom{m-m/z}{\ell}
        \left(\frac{1}{1-p}\right)^\ell
      \right]^{-1},
  \]
  for all complex $z$ such that $0 < \lvert z \rvert < 1 + \delta_{m,p}$.
\end{theorem}

To prove Theorem~\ref{thm:pgf}, we begin by expressing certain events defined by
the sequence $(C_n,U_n)_{n\in\NN}$ in terms of appropriate languages.
Fix $m\in\NN$, and define the alphabet $A \coloneq
\{ \C_1, \ldots,\C_m, \D_1, \ldots, \D_m \}$.
Let $W \coloneq \bigcup_{n \geq 0} A^n$ be the set of finite words using letters from $A$.

We now associate words in $W$ to events with respect to the
coupling in Definition~\ref{def:coupling}.
For each $p\in[0,1)$ and $i\in[m]$, let $\sC_{i}^{p}\coloneq \{i\}\times (0,1-p)$
and $\sD_{i}^{p}\coloneq \{i\}\times[1-p,1)$.
Then for each non-empty $w\in W$ and integer $1\leq k \leq \lvert w \rvert$, set
\[
  \sA_{k}^{p}(w) \coloneq  \begin{cases}
    \sC_{i}^{p}, & w(k) = \C_i,\\
    \sD_{i}^{p}, & w(k) = \D_i,\\
  \end{cases}
\] 
where $w(k)$ denotes the $k$th letter in the word $w$.
Given $p\in[0,1)$, to each word $w\in W$ we then associate the event 
\[
\sE_p(w) \coloneq \bigcap_{k=1}^{\lvert w \rvert}\{(C_k,U_k)\in\sA_{k}^{p}(w)\}.
\]
The letter $\C_i$ therefore encodes the successful collection of a type $i$
coupon, while $\D_i$ encodes the clumsy loss of all type $i$ coupons.
For example, the event illustrated in Figure~\ref{fig:example} corresponds
to the word $\C_1\C_3\C_1\D_1\C_2\C_1$. Therefore,
$\sE_p(\C_1\C_3\C_1\D_1\C_2\C_1) \subseteq \{ T_{m,p} = 6 \}$.
 
For each $p\in[0,1)$, define a probability mass function on $A$
by $\pi_p(\C_i) \coloneq (1-p)/m$ and $\pi_p(\D_i) \coloneq p/m$,
and extend $\pi_p$ multiplicatively to \(W\) as described in Section~\ref{language}.
Then, for all \(w\in W\),
\[
\PP[\sE_p(w)] = \pi_p(w).
\]

We now characterise the events $\{T_{m,p}=n\}$ in terms of languages. 
Define the following languages:
\begin{alignat*}{2}
  H &\coloneq \Bigg\{ w \in W \ &\Bigg|
    &\small\begin{array}{l}
      \text{for all $i \in [m]$,}\\
      \text{there exists an $\ell$ for which $w(\ell) = \C_i$, and}\\
      \text{whenever $w(k) = \D_i$, there exists $k'>k$ such that $w(k')=\C_i$}
    \end{array} \Bigg\},\\
  G &\coloneq \Bigg\{ w \in W \ &\Bigg| \small
    &\small\begin{array}{l}
      \text{for all $i \in [m]$,}\\
      \text{whenever $w(k) = \D_i$, there exists $k'>k$ such that $w(k')=\C_i$}      
    \end{array} \Bigg\},\\
  J &\coloneq \Bigg\{ w \in H \ &\Bigg| \small
    &\small\begin{array}{l}
      \text{whenever $w = u.v$ for $u \in H$ and $v \in W$,
      it follows $v = \varepsilon$}
    \end{array} \Bigg\}.
\end{alignat*}
Thus, $H$ consists of all words $w \in W$ such that, on the event $\sE_p(w)$,
every urn is non-empty at time $\lvert w\rvert$.
In particular, $\sE_p(w) \subseteq \{T_{m,p} \leq \lvert w \rvert\}$ for all $w\in H$. 
The language $G$ consists of all words $w \in W$ such that, on $\sE_p(w)$,
whenever an urn is emptied, a coupon is subsequently inserted into that same urn by time $\lvert w\rvert$.
Finally, $J$ consists of all words $w\in H$ such that, on $\sE_p(w)$,
the first time at which every urn is non-empty is $|w|$.
Consequently, for each $n\in\mathbb N$,
the event $\{T_{m,p}=n\}$ is the disjoint union of the events
$\mathcal E_p(w)$ over all $w\in J$ with $|w|=n$. Therefore
\[
  \PP(T_{m,p} = n) = \sum_{\substack{w \in J\\
  \lvert w \rvert = n}} \pi_p(w)
  = \left[z^n\right]\phi_J(z).
\]
  It then follows that
  \begin{equation}
    \label{eq:lgf is pgf}
    G_{T(m,p)}(z) = \phi_J(z).
  \end{equation}
  We note that the language-generating functions defined via the weights
  $\pi_p$ clearly depend on both $m$ and $p$. However,
  in the interests of readability, we suppress this dependence in our notation.

Our first task is to show that $J$ and $G$ are well-defined for concatenation,
and that $H=J.G$. It will then follow from Proposition~\ref{thm:concatenation}
that we can obtain an expression for $\phi_J$ by computing $\phi_G$ and $\phi_H$.
\begin{figure}
  \centering
  \includegraphics{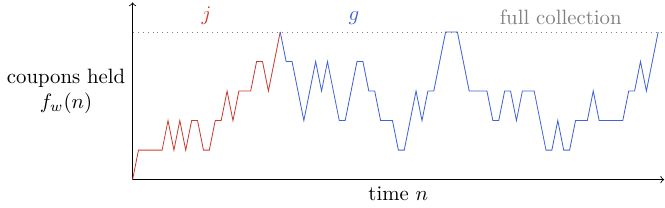}
  \caption{The coupon-counting function $f_w$ for a word $w\in H$. Let $k$ be
  the first time at which $f_w(k)=m$, and write $w=j.g$, where $j$ (shown in red)
  is the length-$k$ prefix and $g$ (shown in blue) is the remaining suffix.
  Lemma~\ref{thm:H decomp} shows that $j\in\mathcal J$ and $g\in\mathcal G$.}
\label{fig:trajectory}
\end{figure}

\begin{lemma}
\label{thm:H decomp}
The language $H$ equals the well-defined concatenation $J.G$.
\begin{proof}
    We begin by showing that $J.G \subseteq H$. 
    Let $j \in J$ and $g \in G$ and $i \in [m]$, and note the following two properties of $j$.
    Firstly, $j$ contains the letter $\C_i$ because $j \in H$.
    Secondly, whenever $j(k) = \D_i$ there exists $\ell > k$ such that $j(\ell) = \C_i$.
    The word $g$ also satisfies the second property, and so $j.g$ satisfies both properties.
    Since $i$ was arbitrary, we conclude that $j.g \in H$ and so $J.G\subseteq H$.
  
    We now proceed to show that each $w\in H$ can be decomposed uniquely as $w=j.g$ with $j\in J$ and $g\in G$. 
    We begin with some definitions. For $w \in W$, let $\# w$ denote the number of coupon types
    held after the sequence of updates encoded by $w$. More precisely 
    \[
      \# w \coloneq \left|\left\{
        i \in [m] \middle| \small
        \begin{array}{l}
          \text{there exists an $\ell$ for which $w(\ell) = \C_i$, and}\\
          \text{whenever $w(k) = \D_i$, it follows $w(k') = \C_i$
          for some $k' > k$}
        \end{array}
        \right\}\right|.
    \]
    Let $w_{[k]}$ be its length-$k$ prefix and let $w_{[-k]}$ be its length-$k$ suffix, with $w_{[0]} \coloneq \varepsilon$.
    We associate to each $w\in W$ a \emph{coupon-counting function} $f_w$,
    \[
      f_w(k) \coloneq \# \left(w_{[k]}\right),
      \quad k = 0, 1, \dots, \lvert w \rvert.
    \]
    Thus, $f_w(k)$ is the number of coupon types held at time $k$,
    upon conditioning on the event $\sE_p(w_{[k]})$.
    Consequently, $w \in H$ if and only if $f_w(\lvert w \rvert) = m$.
    See Figure~\ref{fig:trajectory} for a plot of a coupon-counting function.

    Now fix a word $w \in H$, and let $n=\lvert w \rvert$.
    The coupon-counting function corresponding to $w$ then satisfies $f_w(0) = 0$ and $f_w(n) = m$.
    Let $k$ denote the smallest integer for which $f_w(k) = m$, and note the decomposition $w = w_{[k]}.w_{[-(n-k)]}$.
    Clearly $w_{[k]} \in H$. But by the minimality of $k$, it is not possible to write
    $w_{[k]} = u.v$ with $\# u = m$ and $v \neq \varepsilon$. And so, in fact, we have $w_{[k]} \in J$.
    
    We will now argue that $w_{[-(n-k)]} \in G$.
    Let $i\in[m]$, and suppose there exists $1\leq\ell\leq n-k$ such that $w_{[-(n-k)]}(\ell)=\D_i$.
    This implies $w(k + \ell) = \D_i$.
    Consequently, since $f_w(n) = m$, there exists an integer $\ell' > \ell$ for which $w(k + \ell') = \C_i$.
    Since this holds for all occurrences of $\D_i$ in $w_{[-(n-k)]}$, for any $i\in[m]$, we conclude that $w_{[-(n-k)]}\in G$.
    We have therefore shown that $w$ can be decomposed into $w = w_{[k]}.w_{[-(n-k)]}$ with
    $w_{[k]}\in J$ and $w_{[-(n-k)]}\in G$.

    Finally, we prove that the decomposition is unique.
    The minimality of $k$ implies that for all $\ell < k$ we have $w_{[\ell]}\not\in H$, and therefore $w_{[\ell]}\not \in J$. 
    Suppose there exists $\ell > k$ such that $w_{[\ell]}\in H$. Then $w_{[\ell]}=w_{[k]}.v$ for some nonempty word $v$.
    But since $w_{[k]} \in H$ and $v\neq \varepsilon$ we conclude that $w_{[\ell]}\not\in J$.
    Therefore, $w_{[k]}$ is the only prefix of $w$ which belongs to $J$, and so $w=w_{[k]}.w_{[-(n-k)]}$ is the unique
    decomposition of $w$ into $w=j.g$ with $j\in J$ and $g\in G$. Since this holds for all $w\in H$,
    we conclude that $J$ and $G$ are well-defined for concatenation and that $H\subseteq J.G$.
  \end{proof}
\end{lemma}

We now turn our attention to constructing the generating functions of $G$ and $H$. To this end we introduce
the following \emph{primitive languages}: 
\begin{equation*}
    \C_i^{\geq k} \coloneq \{ \C_i^k, \C_i^{k+1}, \ldots\},\qquad\qquad
    \D_i^{\geq k} \coloneq \{ \D_i^k, \D_i^{k+1}, \ldots\},
\end{equation*}
where an exponent $k$ refers to $k$-fold concatenation,
and $a^0 \coloneq \varepsilon$ for any $a \in A$.
Expressions of the form ``$\C_i.L$'' are shorthand for $\{ \C_i \}.L$. 
The exponential generating functions of these primitive languages are
\begin{equation}
  \label{eq: primitive egfs}
  \begin{split}
  \hat\phi_{\C_i^{\geq k}}(z) &= \sum_{n=k}^\infty \frac{1}{n!} \left(\frac{z\left(1-p\right)}{m}\right)^n, \\
    \hat\phi_{\D_i^{\geq k}}(z) &= \sum_{n=k}^\infty \frac{1}{n!} \left(\frac{z\,p}{m}\right)^n.
  \end{split}
\end{equation}

\begin{proposition}
\label{thm:H egf}
The exponential generating function of the language $H$ is
\[
  \hat\phi_H(z) = \left(1-p\right)^m \left(e^{z/m} - 1\right)^m.
\]
\begin{proof}
The language $H$ can be expressed explicitly using primitive languages as follows
\begin{equation}
  \label{eq:H decomposition}
  H = \Big[\Big(\C_1^{\geq 0} \circ \D_1^{\geq 0}\Big).\C_1\Big]
  \circ \cdots \circ
  \Big[
    \Big(\C_m^{\geq 0} \circ \D_m^{\geq 0}\Big).\C_m
  \Big].
\end{equation}

Fix a coupon type $i \in [m]$.
The shuffling $\C_i^{\geq 0} \circ \D_i^{\geq 0}$ is well-defined
because the two languages use distinct letters. Thus, by
Proposition~\ref{thm:shuffling} and Equation~\eqref{eq: primitive egfs} we have
\[
  \hat\phi_{{\C_i^{\geq 0}} \circ {\D_i^{\geq 0}}}(z)
  = \hat\phi_{\C_i^{\geq 0}}(z) \,
    \hat\phi_{\D_i^{\geq 0}}(z)
  = e^{z/m}.
\]
Then, for any $x \in (-1, 1)$, Proposition~\ref{thm:LB transform} implies
\[
  \phi_{{\C_i^{\geq 0}} \circ {\D_i^{\geq 0}}}(x)
  = \int_0^\infty e^{xt/m}\, e^{-t} \, \diff t
  = \int_0^\infty e^{-\left(1-x/m\right)t} \, \diff t
  = \frac{1}{1-x/m}.
\]
Since, by definition, $\phi_{{\C_i^{\geq 0}} \circ {\D_i^{\geq 0}}}$
  is a power series around $0$ with radius of convergence at least $1$, it follows that
\[
  \phi_{{\C_i^{\geq 0}} \circ {\D_i^{\geq 0}}}(z)
  = \frac{1}{1-z/m},
  \qquad z \in \mathbb{D}.
\]
Similarly, $({\C_i^{\geq0}} \circ {\D_i^{\geq0}}).\C_i$ is a well-defined concatenation because the second language
$\C_i = \{ \C_i \}$ contains only one word.
Thus, by Proposition~\ref{thm:concatenation},
\[
  \phi_{\left({\C_i^{\geq0}} \circ {\D_i^{\geq0}}\right).\C_i}(z)
    =  \frac{1}{1-z/m} \cdot \frac{\left(1-p\right)z}{m}
  = (1-p) \sum_{n=1}^\infty \left(\frac{z}{m}\right)^n.
\]
The corresponding exponential-generating function is therefore
\begin{equation}
  \label{eq:second last step}
  \hat\phi_{\left(\C_i^{\geq0} \circ \D_i^{\geq0}\right).\C_i}(z)
  = (1-p)\sum_{n=1}^\infty \frac{(z/m)^n}{n!}
  = (1-p)\left(e^{z/m} - 1\right),\qquad z\in\CC.
\end{equation}
It then follows from repeated application of Proposition~\ref{thm:shuffling} together with~\eqref{eq:second last step} that
\[
  \hat\phi_H(z) =
  \prod_{i=1}^m \hat\phi_{\left(\C_i^{\geq0} \circ \D_i^{\geq0}\right).\C_i}(z)
  = (1-p)^m\left(e^{z/m} - 1\right)^m. \qedhere
\]
\end{proof}
\end{proposition}

\begin{proposition}
\label{thm:G egf}
The exponential generating function of the language $G$ is
\[
  \hat\phi_G(z) = \left[ (1-p) \left( e^{z/m} - 1 \right) + 1 \right]^m.
\]
\begin{proof}
This proof proceeds much like the proof of
Proposition~\ref{thm:H egf}.
The language $G$ can be expressed explicitly using primitive languages as follows
  \[
  G = \Big[\Big(\C_1^{\geq 0} \circ \D_1^{\geq 1}\Big).\C_1
    \cup \C_1^{\geq0} \Big]
  \circ \cdots \circ
  \Big[
    \Big(\C_m^{\geq 0} \circ \D_m^{\geq 1}\Big).\C_m
    \cup \C_m^{\geq0} \Big].
  \]
Every concatenation, shuffle and union is well-defined in this decomposition,
implying we can freely invoke Propositions~\ref{thm:concatenation},
\ref{thm:shuffling} and~\ref{thm:union}.

Let $i \in [m]$. By Proposition~\ref{thm:shuffling},
  \begin{align}
    \notag
    \hat\phi_{{\C_i^{\geq0}} \circ {\D_i^{\geq1}}}(z)
    &= e^{(1-p)z/m} \left( e^{pz/m} - 1 \right)\\
    \notag
    &= e^{z/m} - e^{(1-p)z/m}\\
    \notag
    &= \sum_{n=0}^\infty
        \frac{(z/m)^n - ((1-p)z/m)^n}{n!},
    \intertext{and so, arguing as in the proof of Proposition~\ref{thm:H egf}, we obtain} 
    \label{eq:phi Ci0 Di1}
    \phi_{{\C_i^{\geq0}} \circ {\D_i^{\geq1}}}(z)
    &= \sum_{n=0}^\infty \left[\left(\frac{z}{m}\right)^n - \left(\frac{\left(1-p\right)z}{m}\right)^n \right], \qquad z\in\mathbb{D}.
  \end{align}
  Then, we invoke Proposition~\ref{thm:concatenation}
  and multiply \eqref{eq:phi Ci0 Di1} by
  $\phi_{\C_i}(z) = \left(1-p\right)z\,m^{-1}$ to obtain
  \begin{align*}
    \phi_{\left({\C_i^{\geq0}} \circ {\D_i^{\geq1}}\right).\C_i}(z)
    &= \frac{\left(1-p\right)z}{m}
      \sum_{n=0}^\infty \left[\left(\frac{z}{m}\right)^n - \left(\frac{\left(1-p\right)z}{m}\right)^n \right]\\
    &= (1-p)\sum_{n=1}^\infty
        (z/m)^{n}
      - \sum_{n=1}^\infty
      ((1-p)\,z/m)^{n}.
    \intertext{Thus,}
    \hat\phi_{\left({\C_i^{\geq0}} \circ {\D_i^{\geq1}}\right).\C_i}(z)
    &= (1-p)\sum_{n=1}^\infty
        \frac{(z/m)^{n}}{n!}
      - \sum_{n=1}^\infty
        \frac{((1-p)z/m)^{n}}{n!}\\
    &= (1-p)\left(e^{z/m} - 1\right)
    - e^{(1-p)z/m} + 1.
  \end{align*}
  Then by Proposition~\ref{thm:union}, we add $\hat\phi_{\C_i^{\geq 0}}(z) = e^{(1-p)z/m}$ to obtain
  \[
    \hat\phi_{
      {\left({\C_i^{\geq0}} \circ {\D_i^{\geq1}}\right).\C_i}
      \cup {\C_i}
    }(z)
    = (1-p) \left( e^{z/m} - 1 \right) + 1.
  \]
  Repeated application of Proposition~\ref{thm:shuffling} then yields
  \[
    \hat\phi_G(z)
    = \left[(1-p) \left( e^{z/m} - 1 \right) + 1\right]^m. \qedhere
  \]
\end{proof}
\end{proposition}

By applying Proposition~\ref{thm:LB transform} to the above expressions for the exponential-generating
functions for $G$ and $H$, we can now obtain explicit expressions for their ordinary generating functions. 

Let $\ell \geq 0$, and recall the definition of the $\ell$th falling factorial
\begin{gather*}
  z^{\underline{\ell}} \coloneq
  z\left(z-1\right)\cdots\left(z-\ell+1\right),
  \qquad z \in \CC,
\end{gather*}
where $z^{\underline 0} = 1$.
The $\ell$th binomial coefficient $\binom{z}{\ell}$ is then given by
\[
  \binom{z}{\ell} \coloneq \frac{1}{\ell!} \, z^{\underline{\ell}},\qquad
  z \in \CC.
\]
\begin{lemma}
  \label{thm:H ogf}
  The ordinary generating function of $H$ is
  \[
    \phi_H(z) = (1-p)^m \binom{m/z - 1}{m}^{-1},
    \qquad 0 < \lvert z \rvert < 1,
  \]
  with $\phi_H(0) = 0$.
  \begin{proof}
    Let $x \in (-1, 1)$.
    Propositions~\ref{thm:LB transform} and~\ref{thm:H egf}, together with $m$ applications of integration by parts, yield
    \begin{align}
      \notag
      \frac{\phi_H(x)}{(1-p)^m} &= \int_0^\infty
        \left( e^{xt/m} - 1 \right)^m
        e^{-t} \, \diff t\\
      \notag
      &= \left. \left( e^{xt/m} - 1 \right)^m \left(-e^{-t}\right) \right|^{t=\infty}_{t=0}\\
      \notag
      &\quad\, + \int_0^\infty m\left( e^{xt/m} - 1 \right)^{m-1} \left(\frac{x}{m} \, e^{xt/m}\right)
      e^{-t} \, \diff t\\
      \notag
      &= \frac{m\left(x/m\right)}{1} \int_0^\infty \left(e^{xt/m}-1\right)^{m-1}
            e^{-(1-x/m)t} \, \diff t\\
      \notag
      &= \frac{m \left(m-1\right) (x/m)^2}{1-x/m} \int_0^\infty \left( e^{xt/m}-1\right)^{m-2}
            e^{-(1-2x/m)t} \, \diff t\\
      \notag
        &\vdotswithin{=}\\
      \notag
      &=
    \frac{m!\,(x/m)^m}{\prod_{\ell=1}^{m}(1-\ell \,x/m)} \\
      &=
        \frac{m!\,x^m}{\prod_{\ell=1}^{m}(m-\ell \,x)}.
      \label{eq:phiH inter}                                  
    \end{align}
    Therefore, defining
    \[
    h(z)\coloneqq\frac{m!\,z^m}{\prod_{\ell=1}^{m}(m-\ell \,z)}, \qquad z\in \mathbb{D},
    \]
    we have $\phi_H(x)=(1-p)^m\, h(x)$ for all $x\in(-1,1)$.
    But $\phi_H$ and $h$ are analytic on $\mathbb{D}$, and so
    in fact $\phi_H(z)=(1-p)^m\, h(z)$ for all $z\in \mathbb{D}$.
    The stated result now follows by observing that $h(0)=0$ and
    \[
      h(z) = \binom{m/z - 1}{m}^{-1}, \qquad z \in \mathbb{D}\setminus \{ 0 \}.
      \qedhere
    \]
  \end{proof}
\end{lemma}

\begin{lemma}
  \label{thm:G ogf}
  The ordinary generating function of $G$ is
  \[
      \phi_G(z)
      = \sum_{\ell=0}^m \frac{m^{\underline{\ell}}
                \left(1-p\right)^\ell}{\left(m/z - 1\right)^{\underline{\ell}}}
      = \sum_{\ell=0}^m \binom{m}{\ell} \binom{m/z-1}{\ell}^{-1} \left(1-p\right)^\ell
  \]
  for all $0 < \lvert z \rvert < 1$, with $\phi_G(0) = 1$.
  \begin{proof}
    The proof mimics that of Lemma~\ref{thm:H ogf}.
    For any $k \in \{ 1, 2, \ldots, m \}$ and $x \in (-1,1)$ it follows, by integration by parts, that
    \begin{equation}
    \label{eq:phiG integration by parts}
    \begin{aligned}
      &\phantom{{}={}} \int_0^\infty \left[ (1-p)\left(e^{xt/m} - 1\right)
                                  + 1\right]^k e^{-(1-(m-k)x/m)t} \, \diff t\\
      &= \frac{k\left(1-p\right) x}{m \left(1 - \left(m-k\right) x/m \right)}
      \int_0^\infty \left[(1-p)\left(e^{xt/m} - 1\right) + 1\right]^{k-1}
      e^{-(1-(m-k+1) x/m)t} \, \diff t\\
      &\quad\, + \frac{1}{1-\left(m-k\right)x/m}.
    \end{aligned}
    \end{equation}
    Propositions~\ref{thm:LB transform} and~\ref{thm:G egf}, together with $m$
    applications of~\eqref{eq:phiG integration by parts}, then imply
    \begin{align*}
      \phi_G(x) &= \int_0^\infty
        \left[ (1-p) \left( e^{xt/m} - 1 \right) + 1 \right]^m
        e^{-t} \, \diff t\\
        &= \sum_{\ell=0}^m \frac{m^{\underline{\ell}} \left(1-p\right)^\ell x^\ell
        }{
          m^\ell \prod_{k=1}^{\ell} (1-k\,x/m)}, \qquad -1 < x < 1.
    \end{align*}
    Therefore, defining
    \[
      g(z)\coloneqq\sum_{\ell=0}^m \frac{m^{\underline{\ell}}
      \left(1-p\right)^\ell z^\ell}{m^\ell \prod_{k=1}^{\ell} (1-k\,z/m)},\qquad  z\in\mathbb{D},
    \]
    we have $\phi_G(x)=g(x)$ for all $-1<x<1$. But $\phi_G$ and $g$ are analytic
    on $\mathbb{D}$, and so in fact $\phi_G(z)=g(z)$ for all $z\in\mathbb{D}$.
    The stated result now follows by observing that $g(0)=1$, while for $z\in\mathbb{D}\setminus \{0\}$
    \begin{align*}
    g(z) &= \sum_{\ell=0}^m \frac{m^{\underline{\ell}} \left(1-p\right)^\ell}{(m/z - 1)^{\underline{\ell}}}\\
         &= \sum_{\ell = 0}^m \frac{\binom{m}{\ell}}{\binom{m/z-1}{\ell}} (1-p)^\ell.
         \qedhere
    \end{align*}
  \end{proof}
\end{lemma}

Armed with Lemmas~\ref{thm:H ogf} and~\ref{thm:G ogf}, we now prove Theorem~\ref{thm:pgf}.
\begin{proof}[Proof of Theorem~\ref{thm:pgf}]
  Proposition~\ref{thm:concatenation} and Lemma~\ref{thm:H decomp} imply $\phi_{H}=\phi_J\,\phi_G$.
  Since $\phi_H(z)\neq 0$ for all $z\in\mathbb{D}\setminus\{0\}$, it then follows that $\phi_G(z)\neq 0$ for all
  $z\in\mathbb{D}\setminus\{0\}$. But Lemma~\ref{thm:G ogf} states $\phi_G(0)=1$,
  and so in fact $\phi_G(z)\neq0$ for all $z\in\mathbb{D}$. 
  It then follows from Equation~\eqref{eq:lgf is pgf} that
  \begin{equation}\label{eq:mgf=H/G}
  G_{T(m, p)}(z) = \frac{\phi_H(z)}{\phi_G(z)}, \qquad z\in\mathbb{D}.
  \end{equation}
  Using Lemmas~\ref{thm:H ogf} and~\ref{thm:G ogf} we therefore have, for all $z\in\mathbb{D}\setminus\{0\}$, that
    \begin{align*}
      G_{T(m, p)}(z) 
      &= \left(1-p\right)^m \binom{m/z - 1}{m}^{-1}
        \left[\sum_{\ell=0}^m \binom{m}{\ell} 
        \binom{m/z-1}{\ell}^{-1} \left(1-p\right)^\ell\right]^{-1} \\
      &= \left(1-p\right)^m
        \left[ \sum_{\ell=0}^m
        \binom{m/z - 1}{m}
        \binom{m}{\ell} 
        \binom{m/z-1}{\ell}^{-1} \left(1-p\right)^\ell\right]^{-1} \\
      &= (1-p)^m
        \left[
        \sum_{\ell=0}^m
        \frac{\left(m/z-1\right)^{\underline{m}} m^{\underline{\ell}}
        }{
          m! \left(m/z-1\right)^{\underline{\ell}}}
          \left(1-p\right)^\ell\right]^{-1}\\
      &= (1-p)^m
        \left[
        \sum_{\ell=0}^m
        \binom{m/z - \ell - 1}{m - \ell}
        \left(1-p\right)^\ell\right]^{-1}.
    \end{align*}
    Changing summation variables $\ell \mapsto m - \ell$,
    and using the identity $\binom{\nu + \ell - 1}{\ell} = (-1)^\ell\binom{-\nu}{\ell}$, 
    valid for all $\nu\in\CC$, we then conclude that for all $z \in \mathbb{D} \setminus \{ 0 \}$
    \begin{equation}
      \label{eq:G as binomials}
      G_{T(m,p)}(z)
      = \left[\sum_{\ell=0}^m(-1)^\ell\binom{m-m/z}{\ell} \left(\frac{1}{1-p}\right)^\ell \right]^{-1}.
    \end{equation}

    Next, we prove that the radius of convergence of $G_{T(m,p)}$ is greater than $1$.
    If $(m,p)=(1,0)$ then $G_{T(m,p)}(z)=z$ for all $z$, and the result holds trivially.
    Therefore, suppose $(m,p)\neq(1,0)$, and
    let $E_k$ be the event that coupons $1, 2, \dots, m$ are respectively collected
    on the days $k\,m + 1, k\,m+ 2, \dots, k\,m + m$.
    Then $\{ E_k \}_{k \geq 0}$ is a collection of
    independent events, each with positive probability
    $q \coloneq ((1-p)/m)^m$.
    The random variable $X \coloneq \inf\{ k \geq 0 \mid \1_{E_k} = 1 \}$
    is geometric with parameter $q$. Thus, for any $n\in\NN$
    \[
      \PP(T_{m, p} = n) \leq \PP(T_{m, p} \geq n)
      \leq \PP(m\,X + m \geq n)
      = (1-q)^{\left\lceil\frac{n-m}{m}\right\rceil}.
    \]
    Setting $\delta_{m,p} \coloneq (1-q)^{-1/m} - 1$, it then follows that
    $G_{T(m,p)}$ has radius of convergence at least $1+\delta_{m,p}$ with $\delta_{m,p}>0$. 
    
    Finally, define
    \begin{align*}
      f(z)&\coloneqq
           z^m\left[
            \sum_{\ell=0}^m \frac{(-1)^\ell}{\ell!} \left(\frac{1}{1-p}\right)^\ell z^{m-\ell} \prod_{k=0}^{\ell-1} (m\,z-m-k\,z)
            \right]^{-1},
    \end{align*}
    and note that~\eqref{eq:G as binomials} implies $G_{T(m,p)}=f(z)$ for all $0<\lvert z\rvert <1$.
    Observe also that $f(z)$ is rational, and has no pole at $z=0$.
    Continuity at $z=0$ then implies $f(z)=G_{T(m,p)}$ for all $z\in\mathbb{D}$.
    Consequently, the power series defining $G_{T(m,p)}$ is the Taylor series of $f$ around $z=0$.
    We conclude that 
    $G_{T(m, p)}(z)=f(z)$ for all $\{ z \in \CC : \lvert z \rvert < 1 + \delta_{m,p} \}$.
\end{proof}

\subsection{Consequences of the probability-generating function}
The characteristic function and moment-generating function of $T_{m,p}$ follow
immediately from Theorem~\ref{thm:pgf} by taking $z=e^{it}$ and $z=e^t$, respectively. 
\begin{corollary}
  For all $p\in[0,1)$ and $m\in\NN$, the characteristic function of $T_{m,p}$ is
  \[
    \varphi_{T(m,p)}(t) \coloneq
    \E\left[ \exp(t\,T_{m,p}\,i) \right]
    = \left[
        \sum_{\ell=0}^m
        (-1)^\ell
        \binom{m-m\,e^{-ti}}{\ell}
        \left(\frac{1}{1-p}\right)^\ell
      \right]^{-1}, \qquad t \in \RR.
  \]
\end{corollary}
\begin{corollary}
  \label{thm:mgf}
  For all $p\in[0,1)$ and $m\in\NN$, there exists $\delta_{m,p}' > 0$ such
  that the moment-generating function of
  $T_{m,p}$ satisfies 
  \[
    M_{T(m,p)}(t) \coloneq
    \E\left[\exp(t\,T_{m,p})\right]
    = \left[
        \sum_{\ell=0}^m
        (-1)^\ell
        \binom{m-m\,e^{-t}}{\ell}
        \left(\frac{1}{1-p}\right)^\ell
      \right]^{-1}, \quad t \in \left(-\infty, \delta_{m,p}'\right).
  \]
\end{corollary}

Corollary~\ref{thm:mgf} can be used to establish exact expressions for the
mean and variance of $T_{m,p}$, valid for any $m\in\NN$ and $p\in[0,1)$. First recall the following definition. 
For non-negative integers $\ell$ and $k$, the unsigned Stirling number of
the first kind $\stirling{\ell}{k}$ is defined, 
for $\ell>0$ and $k\le \ell$, by the expansion
\begin{equation*}
  \label{eq:Stirling}
  z \left(z + 1\right) \cdots \left(z + \ell - 1\right) = \sum_{k=0}^\ell \stirling{\ell}{k} z^k,
\end{equation*}
while $\stirling{\ell}{k}\coloneqq 0$ for $k>\ell$, and $\stirling{0}{0}\coloneqq 1$.
We then have the following useful identity, valid for all $\nu\in\CC$, and all non-negative integers $\ell$
\begin{equation}
  \label{eq:binomials in terms of stirlings}
(-1)^\ell \binom{\nu}{\ell}
=
\frac{1}{\ell!}
\sum_{k=0}^{\ell}
\stirling{\ell}{k}(-\nu)^k.
\end{equation}

\begin{theorem}
  \label{thm:expectation}
  For all $p\in[0,1)$ and $m\in\NN$
  \[
    \E T_{m,p} = m \sum_{\ell = 1}^m \frac{1}{\ell\,(1-p)^\ell}.
  \]
\end{theorem}

\begin{proof}[Proof of Theorem~\ref{thm:expectation}]
Define $\Lambda(t) \coloneq 1/M_{T(m,p)}(t)$ for $t <\delta_{m,p}'$.
Applying~\eqref{eq:binomials in terms of stirlings} to Corollary~\ref{thm:mgf} yields
\begin{align*}
  \Lambda(t) &= \sum_{\ell=0}^m (-1)^{\ell} \binom{m-m\,e^{-t}}{\ell} (1-p)^{-\ell}\\
  &= \sum_{\ell=0}^m \frac{1}{\ell!\,(1-p)^\ell} \sum_{k=0}^\ell \stirling{\ell}{k}
        m^k \left(e^{-t} - 1\right)^k.
\end{align*}
Differentiation yields 
\begin{align}
  \label{eq:Lambda first deriv}
  \Lambda'(t) &= \sum_{\ell=1}^m \frac{1}{\ell!\,(1-p)^\ell}
  \sum_{k=1}^\ell \stirling{\ell}{k}
  k\,m^k \left(e^{-t} - 1\right)^{k-1} \left(-e^{-t}\right),
  \\
  \intertext{and so, using the fact that $\stirling{\ell}{1} = (\ell-1)!$ for $\ell \geq 1$, we obtain}  
  \notag
  \Lambda'(0) &= (-m) \sum_{\ell=1}^m \frac{1}{\ell \,(1-p)^\ell}.
\end{align}
We also note that $\Lambda(0)=1$.
Therefore, since $M_{T(m,p)}(t)$ is finite on an open interval containing $0$, we conclude that
\begin{equation*}
  \E T_{m,p} = M_{T(m,p)}'(0) = -\frac{\Lambda'(0)}{\Lambda^2(0)} =
  m \sum_{\ell=1}^m \frac{1}{\ell \,(1-p)^\ell}. \qedhere
\end{equation*}

\end{proof}

\begin{theorem}
  \label{thm:variance}
  For all $p\in[0,1)$ and $m\in\NN$  
  \[
    \Var(T_{m,p}) = 
	m^2 \sum_{i,j=1}^m \frac{1}{ij\,(1-p)^{i \vee j}}
        \left[ \frac{1}{(1-p)^{i\wedge j}} - 1 \right]
      + \sum_{\ell=1}^m \frac{m}{\ell} \left(\frac{m}{\ell}-1\right)
      \frac{1}{(1-p)^\ell}
  \]
\end{theorem}

\begin{proof}[Proof of Theorem~\ref{thm:variance}]
  Recall the definition of $\Lambda(t)$ from the proof of Theorem~\ref{thm:expectation}.
  Differentiating~\eqref{eq:Lambda first deriv}, we find
\[
  \Lambda''(t)
  = \sum_{\ell=1}^m \frac{1}{\ell!\left(1-p\right)^\ell} \sum_{k=1}^\ell 
  \stirling{\ell}{k} k\,m^k
  \left[\left(k-1\right)\left(e^{-t} - 1\right)^{k-2} \left(-e^{-t}\right)^2
  + \left(e^{-t} - 1\right)^{k-1} e^{-t}\right].
\]
Consequently, using the fact that $\stirling{\ell}{2} = (\ell-1)! \, H_{\ell-1}$ for $\ell\ge2$, we obtain
\begin{align*}
  \Lambda''(0)
  &= \sum_{\ell=1}^m \frac{1}{\ell!\,(1-p)^\ell}
  \left( \stirling{\ell}{1} m + 2 \stirling{\ell}{2} m^2 \right)\\
  &= m \sum_{\ell=1}^m \frac{1}{\ell\,(1-p)^\ell}
    + 2\,m^2 \sum_{\ell=2}^m \frac{H_{\ell-1}}{\ell\,(1-p)^\ell}.
\end{align*}
Expressing the second derivative of $M_{T(m,p)}$ in terms of $\Lambda$ yields
\begin{equation}
  \label{eq:mgf second deriv}
  M_{T(m,p)}''(t)
  = \frac{2\left[\Lambda'(t)\right]^2
      - \Lambda''(t)\,\Lambda(t)}{\Lambda(t)^3}.
\end{equation}
It then follows that
\begin{align}
  \label{eq:variance}
  \Var(T_{m,p}) &= M_{T(m,p)}''(0) - [M_{T(m,p)}'(0)]^2\notag\\
  &= 2\,\Lambda'(0)^2 - \Lambda''(0) - \left( - \Lambda'(0) \right)^2\notag\\
  &= \Lambda'(0)^2 - \Lambda''(0)\notag\\
  &= \left( m \sum_{\ell=1}^m \frac{1}{\ell\,(1-p)^\ell}\right)^2
    - m \sum_{\ell=1}^m \frac{1}{\ell\,(1-p)^\ell}
    - 2 \, m^2 \sum_{\ell=2}^m \frac{H_{\ell-1}}{\ell\,(1-p)^\ell}\notag\\
  &= m^2 \sum_{i,j=1}^m \frac{1}{ij\,(1-p)^{i+j}}
  - m \sum_{\ell=1}^m \frac{1}{\ell\,(1-p)^\ell}
  - 2\, m^2 \sum_{1 \leq i < j \leq m} \frac{1}{ij\,(1-p)^j}.
\end{align}
  Finally, using
  \begin{equation*}
    \sum_{i,j=1}^m
    \frac{1}{ij(1-p)^{i\vee j}}
    =
    2\sum_{1\leq i<j\leq m}\frac{1}{ij\,(1-p)^j}  + \sum_{\ell=1}^m
    \frac{1}{\ell^2\,(1-p)^\ell},
  \end{equation*}  
  we can rearrange~\eqref{eq:variance} into a sum of two non-negative terms:
  \begin{equation*} 
  \Var(T_{m,p})
  = m^2 \sum_{i,j=1}^m \frac{1}{ij} \left[ \frac{1}{(1-p)^{i+j}}
	- \frac{1}{(1-p)^{i\vee j}} \right]
  + \sum_{\ell=1}^m \frac{m}{\ell} \left(\frac{m}{\ell}-1\right)
  \frac{1}{(1-p)^\ell}. \qedhere
\end{equation*}
\end{proof}

\section{Limit theorems}
\label{asymptotics}
In this section we prove Theorems~\ref{thm:subcritical shape},~\ref{thm:supercritical shape} and~\ref{thm:critical shape}.
We begin by introducing the function
\begin{equation}
  \label{eq:extra term}
  I_{m, p}(t) \coloneq
    (1-p)^m
    \left[
      1
      -
      \int_0^1 p \, m
      \left(1-x\right)^{m\exp(-t)-m}
      \left(1-p\,x\right)^{m-1} \diff x
    \right]^{-1},\,\qquad t\le 0.
\end{equation}
Our main tool is the following expression for the moment-generating function.
\begin{lemma}
  \label{thm:factored mgf}
  Let $p\in[0,1)$ and $m\in\NN$. Then for all $t \leq 0$
\[
  M_{T(m,p)}(t)
  = \binom{m\,e^{-t}}{m}^{-1} \, I_{m,p}(t).
\]
\end{lemma}
\begin{proof}
It follows from Theorem~\ref{thm:expectation} that $T_{m,p}<\infty$ almost surely,
  and so $M_{T(m,p)}(0)=1$. Since $I_{m,p}(0)=1$, the stated result holds at $t=0$.
Suppose instead that $t<0$ and set $z=e^{t}$. It follows from Lemma~\ref{thm:G ogf} that 
\begin{align*}
  \phi_G(z)
  &= \sum_{\ell=0}^m \binom{m}{\ell} \binom{m/z - 1}{\ell}^{-1} (1-p)^\ell\\
  &= \sum_{\ell=0}^m \binom{m}{\ell} \left(1-p\right)^\ell
    \frac{\Gammaop(\ell + 1) \Gammaop(m/z-\ell)}{\Gammaop(m/z)}\\
  &= \frac{m}{z} \sum_{\ell=0}^m \binom{m}{\ell} \left(1-p\right)^\ell
    \frac{\Gammaop(\ell + 1) \Gammaop(m/z-\ell)}{\Gammaop(m/z + 1)}\\
  &= \frac{m}{z} \sum_{\ell=0}^m \binom{m}{\ell} \left(1-p\right)^\ell 
    \Betaop(\ell+1,m/z-\ell)
\end{align*}
where $\Betaop(a,b)$ denotes the Euler beta function. The standard integral representation of the beta function then implies
\begin{align*}  
\phi_G(z)  &= \frac{m}{z} \sum_{\ell=0}^m \binom{m}{\ell} \left(1-p\right)^\ell
    \int_0^1 y^\ell \left(1-y\right)^{m/z-\ell-1} \diff y\\
  &= \frac{m}{z} \int_0^1 (1-y)^{m/z-1}
    \sum_{\ell=0}^m \binom{m}{\ell} \left(1-p\right)^\ell
    y^\ell \left(1-y\right)^{-\ell} \diff y\\
  &= \frac{m}{z} \int_0^1 \left(1-y\right)^{m/z-1}
    \left(1 + \frac{\left(1-p\right)y}{1-y}\right)^m \diff y\\
           &= \frac{m}{z} \int_0^1 \left(1-y\right)^{m/z-m-1}\left(1-p\,y\right)^m \diff y\\
  &= \frac{m/z}{m/z-m}\left(1 - \int_0^1 p\,m  \left(1-y\right)^{m/z-m}\left( 1-p\,y \right)^{m-1} \, \diff y \right)
\end{align*}
where the final step follows via an integration by parts. Combining this with Lemma~\ref{thm:H ogf} and~\eqref{eq:mgf=H/G}, we conclude that
\begin{equation*}
\begin{aligned}
  M_{T(m,p)}(t)&=G_{T(m,p)}(z)\\
  &= (1-p)^m \binom{m/z}{m}^{-1}
\left[1 - \int_0^1 p\,m  \left(1-y\right)^{m/z-m}\left( 1-p\,y \right)^{m-1} \, \diff y \right]^{-1}.\qedhere
\end{aligned}
\end{equation*}
\end{proof}

Since $I_{m,0}(t)=1$, we immediately conclude from Lemma~\ref{thm:factored mgf}
that $T_{m, 0}$ has moment-generating function~$1/\binom{m\,e^{-t}}{m}$.
As a simple consequence we recover~\cite{erdos1961coupon} the following result for the classical coupon collector. 
\begin{proposition}
  \label{thm:classical mgf convergence}
  Let $W_m\coloneq (T_{m,0} - m\log{m})/m$. Then for any fixed $t<0$,
  \[
  \lim_{m\to\infty} M_{W_m}(t) = \Gammaop(1-t).
  \]
  \begin{proof}
    Fix $t < 0$.
    Recalling that $\Betaop(\cdot,\cdot)$ denotes the Euler beta function, it follows from
    Lemma~\ref{thm:factored mgf} that
    \begin{align*}
      M_{W_m}(t) &= m^{-t}\,M_{T(m,0)}(t/m) \\
      &= m^{-t} \binom{m \, e^{-t/m}}{m}^{-1} \\
      &= m^{1-t}\Betaop\left(m\,e^{-t/m}-m+1,m\right)\\
      &= m^{1-t} \int_0^1 x^{m\exp(-t/m)-m} \left(1-x\right)^{m-1} \diff x\\
      &= m^{-t+m-m\exp(-t/m)} \int_0^m x^{m\exp(-t/m)-m}\left(1-x/m\right)^{m-1} \, \diff x.
    \end{align*}
    Noting that $m-m\exp(-t/m)=t+O(1/m)$, the stated result now follows by dominated convergence.
  \end{proof}
\end{proposition}

The factorised structure of $M_{T(m,p)}$ is elucidated by the following result.
\begin{lemma}
  \label{thm:coupling}
  The random variable $T_{m,p} - T_{m,0}$ is independent of $T_{m, 0}$, and has moment-generating function
  \[
    M_{T(m,p) - T(m,0)}(t) = I_{m, p}(t), \quad t \leq 0.
  \]
  \begin{proof}
    To begin, we construct an explicit function $f$ such that $T_{m,p} - T_{m,0}$ is the
    image under $f$ of a specific subset of the noise $\{C_i, U_i\}_{i=1}^\infty$.
    To this end, for any $k\in\NN\cup\{\infty\}$, $\mathbf{c}\in[m]^k$ and integer $1\le n \le k$ define
    \begin{equation*}
      [\mathbf{c}]_n\coloneqq \{c_1,\ldots,c_n\},
    \end{equation*}
    and let $[\mathbf{c}]_0\coloneq\emptyset$. Therefore, $[\mathbf{c}]_n$ is the (unordered) set
    of coupon types appearing among the first $n$ entries of $\mathbf{c}$.
    Then, for $k\in\NN$, set
    \begin{equation*}
      \sS_{k} \coloneqq \left\{\mathbf{c}\in[m]^k \,\middle|\, [\mathbf{c}]_{k-1}\neq[m] \text{ and } [\mathbf{c}]_{k}= [m]\right\},            
    \end{equation*}
    and observe that $\{T_{m,0}=k\}=\bigcup_{\mathbf{c}\in\sS_k} \{(C_1,\ldots,C_k)=\mathbf{c}\}$.

    For $k\ge m$, also define
    \[
    \ell(k;i,\mathbf{c}) \coloneq \max\{j\in[k] \mid c_j=i\}, \qquad i\in [m],\, \mathbf{c}\in\sS_k,
    \]
    so that $\ell(k;i,\mathbf{c})$ is the last time, up to time $k$, at which a coupon of type $i$ was issued, with respect to the coupon vector $\mathbf{c}\in \sS_k$.
    And for each coupon type $i\in[m]$, time $n\in\NN$, and coupon sequence $\mathbf{c} \in [m]^\infty$, define
    \[
      \sL(i, n; \mathbf{c}) \coloneq
      \sup\{ k \in[n] \mid c_k = i \}.
    \]
    In the notation of Definition~\ref{def:coupling}, we therefore have $L(i,n)=\sL(i,n; (C_k)_{k\in\NN})$. 
    Finally, for each $\mathbf{c} \in [m]^\infty$, $\mathbf{u}^* \in (0, 1)^m$, and $\mathbf{u}\in (0, 1)^\NN$, define 
    \begin{align*}
      f(\mathbf{u}^*, \mathbf{c}, \mathbf{u}) \coloneqq
      \inf\{ n\ge 0  \mid u^*_i<1-p \text{ for all } i\not\in[\mathbf{c}]_n \text{ and } u_{\sL(i,n;\mathbf{c})}<1-p \text{ for all } i\in[\mathbf{c}]_n \}.
    \end{align*}
 
    Now let $k\ge m$ and $\mathbf{c}\in\sS_k$. On the event $\{(C_1,\ldots,C_k) = \mathbf{c}\}$, we have
    \[
    T_{m,p}-T_{m,0} = f\left((U_{\ell(k;i;\mathbf{c})})_{i\in[m]}, (C_{k+n})_{n\in\NN},  (U_{k+n})_{n\in\NN}\right).
    \]
    Observe also that for any $k\ge m$ and $\mathbf{c}\in\sS_k$, the collections
    $((U_{\ell(k;i;\mathbf{c})})_{i\in[m]},(U_{k+n})_{n\in\NN})$ and $((U_i)_{i\in[m]},(U_{m+n})_{n\in\NN})$ have the same joint distribution.
    It then follows that for any $r\ge0$
    \begin{align}
      \PP\bigl(&T_{m, 0} = k, \,T_{m, p} - T_{m,0} = r\bigr) \nonumber\\
      &= \sum_{\mathbf{c}\in\sS_k} \PP \Biggl(
      \Bigl\{ (C_1, \dots, C_k) = \mathbf{c} \Bigr\}
        \cap \Biggl\{
        f\left(\left(U_{\ell(k;i;\mathbf{c})}\right)_{i\in[m]},\left(C_{k+n}\right)_{n \in\NN},\left(U_{k+n}\right)_{n \in\NN}\right) = r\Biggr\} \Biggr)\nonumber\\
      &= \sum_{\mathbf{c}\in\sS_k} \PP \left(
        \Bigl\{ (C_1, \dots, C_k) = \mathbf{c} \Bigr\}\right)
        \PP \Biggl(
        f\left(\left(U_{\ell(k;i;\mathbf{c})}\right)_{i\in[m]},\left(C_{k+n}\right)_{n \in\NN},\left(U_{k+n}\right)_{n \in\NN}\right) = r  \Biggr)\nonumber\\
      &= \sum_{\mathbf{c}\in\sS_k} \PP \left(
        \Bigl\{ (C_1, \dots, C_k) = \mathbf{c} \Bigr\}\right)
        \PP \Biggl( f\left(\left(U_{i}\right)_{i\in[m]},\left(C_{m+n}\right)_{n \in\NN},\left(U_{m+n}\right)_{n \in\NN}\right) = r  \Biggr)\nonumber \\
      &= \PP(T_{m,0}=k)\,
        \PP \Biggl( f\left(\left(U_{i}\right)_{i\in[m]},\left(C_{m+n}\right)_{n \in\NN},\left(U_{m+n}\right)_{n \in\NN}\right) = r  \Biggr).
        \label{eq: probability factorisation}
    \end{align}
    Note that although~\eqref{eq: probability factorisation} was derived for $k\ge m$, since both sides are $0$ when $k<m$, it in fact holds for all $k\ge0$. 
    And summing over $k\ge0$ we then see, in particular, that
    \[
    \PP\bigl(T_{m, p} - T_{m,0} = r\bigr)=    
    \PP \Biggl( f\left(\left(U_{i}\right)_{i\in[m]},\left(C_{m+n}\right)_{n \in\NN},\left(U_{m+n}\right)_{n \in\NN}\right) = r  \Biggr).
    \]
    Hence, it follows from~\eqref{eq: probability factorisation} that for all $k,r\ge0$
    \[
    \PP\bigl(T_{m, 0} = k, \,T_{m, p} - T_{m,0} = r \bigr) = \PP(T_{m,0}=k)\,\PP\bigl(T_{m, p} - T_{m,0} = r\bigr),
    \]
    which establishes the claimed independence.

    Finally, combining the resulting factorisation $M_{T(m,p)}(t) = M_{T(m,0)}(t)\,M_{T(m,p) - T(m,0)}(t)$ with Lemma~\ref{thm:factored mgf} implies
    \[
      M_{T(m,p) - T(m,0)}(t) = I_{m,p}(t), \qquad t \leq 0. \qedhere
    \]
  \end{proof}
\end{lemma}

We will make use of the following continuity theorem for moment-generating functions, which generalises
the classical result of Curtiss~\cite{Curtiss1942} to intervals which need not contain the origin.
\begin{theorem}[Continuity theorem~\cite{MukherjeaRaoSuen2006}]
  \label{thm:convergence}
  Let $X$ and $X_1, X_2, \dots$ be random variables with moment-generating
  functions $M_X$ and $M_{X_1}, M_{X_2}, \dots$ respectively,
  which are finite on an interval $(a,b)$ with $a<b$.
  If $M_{X_m}$ converges pointwise to $M_X$ on $(a,b)$ as $m \to \infty$, then
  $X_m$ converges weakly to $X$ as $m \to \infty$.
  \begin{proof}
    We consider three cases. 
    In the case that $a>0$, the stated result is precisely~\cite[Theorem 2]{MukherjeaRaoSuen2006}.
    In the case that $a\le 0$ and $b>0$, the stated result follows by applying~\cite[Theorem 2]{MukherjeaRaoSuen2006}
    on the subinterval $(b/2,b)$. Finally, in the case that $b\le0$, we observe that for all $s\in(-b,-a)$
    \[
    M_{-X_m}(s) = M_{X_m}(-s) \xrightarrow[m\to\infty]{} M_X(-s) = M_{-X}(s),
    \]
    and so applying the previous cases to $-X_m$ implies that $-X_{m}$ converges weakly to $-X$.
    It then follows from the mapping theorem~\cite[Theorem 25.7]{Billingsley1995} that $X_m$ converges weakly to $X$. 
  \end{proof}
\end{theorem}

We now have the ingredients to prove Theorem~\ref{thm:subcritical shape}.
\begin{proof}[Proof of Theorem~\ref{thm:subcritical shape}]
  Let $G$ be a standard Gumbel random variable, and recall that its moment-generating function is
  $\Gammaop\left(1- t \right)$ for $t<1$. It follows by combining Proposition~\ref{thm:classical mgf convergence}
  and Theorem~\ref{thm:convergence} that $(T_{m, 0} - m\log{m})/m$ converges weakly to $G$.
  But $T_{m, p} - T_{m, 0} \prob 0$ because $\PP(T_{m, 0} = T_{m, p}) = (1-p)^m \to 1$,
  and so using~\cite[Theorem 25.4]{Billingsley1995} we conclude $(T_{m,p} - m \log{m})/m \weak G$.
\end{proof}

The main additional ingredient needed to study the supercritical case is the following lemma.
\begin{lemma}
  \label{thm:supercritical}
  If $p = \omega(1/m)$ then $p\,(1-p)^m\,(T_{m,p} - T_{m,0})$ converges in law to an exponential distribution with rate $1$.
  \begin{proof}
    It follows from Lemma~\ref{thm:coupling} that $I_{m, p}(p\,(1-p)^m\,t)$ is
    the moment-generating function of $p\,(1-p)^m\,(T_{m,p} - T_{m,0})$, for all $t\le0$.
    Fix $t<0$, and set $\alpha(m) \coloneq m\exp(p\,(1-p)^m\,(-t)) - m$. Then
    \begin{align}\label{eq:I reciprocal}
      \frac{1}{I_{m,p}(p\,(1-p)^m\,t)}
      &= (1-p)^{-m}\left[ 1 - \int_0^1 p\,m
      \left(1-x\right)^{\alpha(m)}
      \left(1-p\,x\right)^{m-1} \diff x \right]\nonumber\\
      &= (1-p)^{-m}\left[ 1 - \int_0^1 p\,m
      \left(1 + \sum_{\ell=1}^\infty
      (-x)^\ell \binom{\alpha(m)}{\ell} \right)
      \left(1-p\,x\right)^{m-1}
      \diff x \right]\nonumber\\
      &= 1 + \left(1-p\right)^{-m} p\,m
      \int_0^1 \left(1-p\,x\right)^{m-1}
      \sum_{\ell=1}^\infty
      \left(-1\right)^{\ell-1} \binom{\alpha(m)}{\ell}
      \, x^\ell \, \diff x.
    \end{align}
    We now establish upper and lower bounds for $I_{m,p}(p\,(1-p)^m\,t)^{-1}$.
    First, observe that $\alpha(m)$ is positive, and converges to $0$
    as $m \to \infty$. Thus, $\alpha(m)$ lies in the interval $(0, 1)$ for all sufficiently large $m$. 
    The falling factorial $\alpha(m)^{\underline{\ell}}$ therefore
    contains one positive factor and $\ell - 1$ negative
    factors. Hence, $(-1)^{\ell-1} \binom{\alpha(m)}{\ell}$
    is positive for all sufficiently large $m$, and therefore
    \begin{align}
      \notag
      \frac{1}{I_{m,p}(p\,(1-p)^m\,t)}
      &\geq 1 + \left(1-p\right)^{-m} p\,m 
      \int_0^1 \left(1-p\,x\right)^{m-1}
      \alpha(m)\, x \, \diff x\\
      \label{eq:I lower bound temp}
      &= 1 + \frac{\alpha(m)}{p\,m\,(1-p)^m}
      \int_0^{p m} \left(1-x/m\right)^{m-1}x \, \diff x.
    \end{align}
    Now note that $\alpha(m) \sim p\,m\,(1-p)^m\,(-t)$ as $m \to \infty$, and
    \[
     \lim_{m\to\infty} \int_0^{p\,m} \left(1-x/m\right)^{m-1}
      x \, \diff x
      = \int_0^\infty e^{-x} \, x \, \diff x\\
      = 1
    \]
    by dominated convergence. Therefore
    \begin{equation}
      \label{eq:I lower bound}
      \frac{1}{I_{m,p}(p\,(1-p)^m\,t)}\ge
      1 + \frac{\alpha(m)}{p\,m\,(1-p)^m}
      \int_0^{p\,m} \left(1-x/m\right)^{m-1}x \, \diff x
      \xrightarrow[m\to\infty]{} 1-t.      
    \end{equation}

    We now turn our attention to the upper bound. 
    For any $\nu \in (0, 1)$ and integer $\ell \geq 1$,
    \begin{equation}
      \label{eq:binom coeff bound}
      \begin{aligned}
        \left(-1\right)^{\ell-1} \binom{\nu}{\ell}
        &= \frac{\nu}{\ell!} \, (1-\nu)\cdots(\ell-1-\nu)\\
        &\leq \frac{\nu}{\ell}.
      \end{aligned}
    \end{equation}
    Applying~\eqref{eq:binom coeff bound} to~\eqref{eq:I reciprocal} implies
    \begin{align}
      \notag
      \frac{1}{I_{m,p}(p\,(1-p)^m\,t)}
      &\leq 1 + \left(1-p\right)^{-m} p\,m
      \int_0^1 \left(1-p\,x\right)^{m-1}
      \sum_{\ell=1}^\infty
      \frac{\alpha(m)}{\ell} \,
      x^\ell \, \diff x\\
      \notag
      &= 1 + t \left(1 + o(1)\right)\left(p\,m\right)^2
      \int_0^1 \left(1-p\,x\right)^{m-1} \log(1-x)\,\diff x\\
      \label{eq:I upper bound temp}
      &\leq 1 + t \left(1 + o(1)\right) \left(p\,m\right)^2
      \int_0^1 e^{-p(m-1)x}\,\log(1-x)\, \diff x.
    \end{align}
    But it follows from Laplace's method~\cite[p.~58]{wong1989asymptotics} that
    \begin{equation}
      \label{eq:int asymp for I upper}
      \int_0^1 e^{-\lambda x} \log(1-x)\, \diff x
      \sim -\frac{1}{\lambda^2}, \qquad     \lambda\to\infty.
    \end{equation}
    Combining~\eqref{eq:int asymp for I upper} and~\eqref{eq:I upper bound temp} we then obtain
    \begin{equation}
      \label{eq:I upper bound}
      \frac{1}{I_{m,p}(p\,(1-p)^m\,t)}
      \leq 1 + \left(1 + o(1)\right) (-t) \left(p\,m\right)^2
      \frac{1 + o(1)}{(p\,(m-1))^2}
      \xrightarrow[m\to\infty]{} 1-t.
    \end{equation}

    We therefore conclude from~\eqref{eq:I lower bound} and~\eqref{eq:I upper bound} that
    $I_{m,p}(p\,(1-p)^m\,t)^{-1}$ converges to $1-t$
    as $m \to \infty$, for all $t < 0$. Let $X$ be a rate-$1$ exponential random variable, and
    recall that the moment-generating function of $X$ is $M_X(t)=1/(1-t)$ for $t<1$.
    It now follows from Lemma~\ref{thm:coupling} and Theorem~\ref{thm:convergence} that
    $p\,(1-p)^m\,(T_{m,p} - T_{m,0})$ converges weakly to $X$ as $m \to \infty$.
  \end{proof}
\end{lemma} 
 
We now have the required ingredients to prove Theorem~\ref{thm:supercritical shape}.
\begin{proof}[Proof of Theorem~\ref{thm:supercritical shape}]
  Let $X$ denote a rate-$1$ exponential random variable, and $G$ a standard Gumbel random variable, and
  let $Y_m\coloneqq p \, (1-p)^m \, (T_{m,p} - m \log{m})$. Setting
  $X_m\coloneqq p\,(1-p)^m \, (T_{m,p} - T_{m,0})$ and
  $Z_m\coloneqq p\,(1-p)^m\,(T_{m, 0} - m\log{m})$ we have $Y_m=X_m+Z_m$.
  But since $a_m\coloneqq mp\,(1-p)^m\to0$ as $m\to\infty$, and $(T_{m,0} - m\log m)/m\weak G$,
  it follows~\cite[Lemma 2, p. 193]{Billingsley1995} that $Z_m=a_m(T_{m,0} - m\log m)/m\weak 0$.
  Then, since Lemma~\ref{thm:supercritical} implies $X_m\weak X$,
  it follows from~\cite[Theorem 25.4]{Billingsley1995} that $Y_m\weak X$.
\end{proof}

Finally, we turn our attention to the critical case. The key additional ingredient is the following. 
\begin{lemma}
  \label{thm:birth--death}
  Let $Q_t$ be the birth--death process defined in Definition~\ref{def:birth-death}
  and let $\tau_c$ be the first hitting time of $0$.
  Then its moment-generating function is
  \[
    \E e^{t\tau_c} = e^{-c}
    \left[1 - c\int_0^1 \left(1-x\right)^{-t} e^{-cx} \, \diff x
    \right]^{-1},
    \qquad t < 0.
  \]
  \begin{proof}
    Fix $t < 0$, and define 
    \begin{equation}
      \label{eq:Phi}
      \Phi_t(y) \coloneq \sum_{n=0}^\infty \phi_n(t) \, \frac{y^n}{n!}, \qquad y\in\RR,
    \end{equation}
    where $\phi_n(t) \coloneq \E(e^{t\tau_c} \mid Q_0 = n)$, and note that $\Phi_t(y)$ is analytic on $\RR$. Then
    \begin{align}
      \label{eq:laplace transform in terms of Phi}
      \E e^{t\tau_c}
      &= \sum_{n = 0}^\infty \E\left(e^{t\tau_c} \middle| Q_0 = n\right) \PP(Q_0 = n)\nonumber\\
      &= e^{-c} \sum_{n = 0}^\infty \frac{c^n}{n!}\, \phi_n(t)\nonumber\\
      &= e^{-c} \, \Phi_t(c).
    \end{align}
    Clearly $\phi_0(t) = 1$, and so $\Phi_t(0)=1$. Moreover, first-step analysis gives
    \[
      (c+n-t) \, \phi_n(t) = 
      n \, \phi_{n-1}(t) + c \, \phi_{n+1}(t),
      \quad n > 0,
    \]
    which implies
    \[
      \sum_{n = 1}^\infty
      \frac{y^n}{n!} \left(c + n -t \right)
      \phi_n(t) = \sum_{n = 1}^\infty
      \frac{y^n}{n!} \left[ n \, \phi_{n-1}(t) + c \, \phi_{n+1}(t)\right].
  \]
  Rearranging, we obtain
  \begin{equation}
    \label{eq:critical ode}
    (c-y)\,\Phi_t'(y) - (c-y-t)\,\Phi_t(y) = c\, \phi_1(t) - c+t,\qquad y\in\RR.
  \end{equation}
  In particular, by setting $y=c$ we see that $t\, \Phi_t(c) = c \phi_1(t) - c +t$. It therefore follows from~\eqref{eq:critical ode}
  that $\Phi_t(y)$ satisfies the following initial-value problem on $[0,c)$
     \begin{equation}
       \label{eq:IVP}
       \Phi_t'(y) - \left(1-\frac{t}{c-y}\right) \Phi_t(y) = \frac{t \,\Phi_t(c)}{c-y}, \qquad \Phi_t(0)=1,
     \end{equation}
     whose unique solution is 
     \begin{equation}
       \label{eq:critical solution}
       \Phi_t(y) = \frac{e^{y}}{(c-y)^{-t}}\left(c^{-t} +t\Phi_t(c) \int_0^y e^{-x} (c-x)^{-t-1} \diff x \right), \qquad y\in[0,c).
     \end{equation}
  Since $\Phi_t(y)$ is everywhere continuous, it now follows from~\eqref{eq:critical solution} that
  \[
  0=\lim_{y\uparrow c} e^{-y}\,(c-y)^{-t}\, \Phi_t(y) = c^{-t} +t\,\Phi_t(c) \int_0^c e^{-x}\,(c-x)^{-t-1}\,\diff x.
  \]
  Rearranging, and performing an integration by parts, we conclude that
  \begin{align}
    \Phi_t(c) &= -\frac{c^{-t}}{\int_0^c e^{-x} \, t \left(c-x\right)^{-t-1} \diff x}\nonumber\\ 
     &= \left[1 - c \int_0^1 \left(1-x\right)^{-t} e^{-cx}\, \diff x \right]^{-1}.\label{eq:final result for Phi}
  \end{align}
  The stated result then follows for all $t<0$ by combining~\eqref{eq:laplace transform in terms of Phi} with~\eqref{eq:final result for Phi}.
  \end{proof}
\end{lemma}

Armed with Lemma~\ref{thm:birth--death}, we now close this section by proving Theorem~\ref{thm:critical shape}.
\begin{proof}[Proof of Theorem~\ref{thm:critical shape}]
  Set $Z_m\coloneqq (T_{m, p} - m\log{m})/m$, $W_m\coloneqq (T_{m,0} - m\log{m})/m$ and
  $Y_m\coloneq (T_{m,p} - T_{m,0})/m$, so that $Z_m=W_m+Y_m$.
  Using Lemmas~\ref{thm:coupling} and~\ref{thm:birth--death} it follows by dominated convergence that for all $t < 0$ 
  \begin{align}
    M_{Y_m}(t)
    &= (1-p)^m \left[1-\int_0^1 p \, m \left(1-x\right)^{m\exp(-t/m)-m} \left(1-p\,x\right)^{m-1} \diff x \right]^{-1}\nonumber\\
    &\xrightarrow[m\to\infty]{}e^{-c} \left[1 - \int_0^1 c \left(1-x\right)^{-t} e^{-c x} \, \diff x \right]^{-1}\nonumber\\
    & = M_{\tau_c}(t).\label{eq:critical difference mgf}
  \end{align}
  Let $G$ be a standard Gumbel random variable, independent of $\tau_c$.
  Lemma~\ref{thm:coupling} implies $W_m$ and $Y_m$ are independent, and so
  combining~\eqref{eq:critical difference mgf} with Proposition~\ref{thm:classical mgf convergence}, it follows that for all $t<0$
  \begin{align*}
    \lim_{m\to\infty} M_{Z_m}(t) &= \lim_{m\to\infty} M_{W_m}(t)\,M_{Y_m}(t)\\
                                 &= M_{G}(t)\, M_{\tau_c}(t)\\
                                 &= M_{G+\tau_c}(t).
  \end{align*}
  We therefore conclude from Theorem~\ref{thm:convergence} that $Z_m\weak G+\tau_c$. 
\end{proof}

\section{Asymptotics of the mean and variance of $T_{m,p}$}
\label{practical}
This section is devoted to proving Theorems~\ref{thm:first moment} and~\ref{thm:second moment},
which respectively describe the large $m$ asymptotics of the mean and variance of $T_{m,p}$.
To begin, however, we derive integral identities which hold for all $m\in\NN$.

To this end, define $L\colon(-\infty,1]\to\RR$ by $L(1)\coloneqq 1$ and $L(x) \coloneqq x + (1-x)\log(1-x)$ for $x<1$.
Similarly, define $N\colon(-\infty,1]\to\RR$ by $N(1)\coloneqq2$ and
\[\
N(x) \coloneqq (\log^2(1-x) - 2\log(1-x)+2)\,(x-1) + 2, \qquad x<1.
\]
Then, for $p\in[0,1)$ and $m\in\NN$ further define
\begin{equation*}
    \label{eq:correction factor}
    \kappa(m, p) \coloneqq
    (p\,m)^2\left(1-p\right)^{m-1}
    +
    \int_0^1 p^3 \, m^2 \left(m-1\right)
    \left(1-p\,x\right)^{m-2} L(x) \, \diff x,
  \end{equation*}
  and
  \begin{equation}
    \label{eq:eta def}
    \eta(m, p) \coloneq
    (1-p)^{-m} \Biggl\{2\,p\,m^3 \left(1-p\right)^{m-1} + \int_0^1 p^2\, m^3 \left(m-1\right) \left(1 - p\,x\right)^{m-2} N(x) \, \diff x
    \Biggr\}.
  \end{equation}
  Finally, recall that $H_m$ denotes the $m$th harmonic number, and $H_m^{(2)}$ denotes the $m$th harmonic number of order $2$.
  With these definitions, we have the following expressions for the mean and variance.
\begin{lemma}
  \label{thm:int first moment}
  For $p\in[0,1)$ and $m\in\NN$ we have 
  \[
    \E T_{m,p} = m \, H_m
    + \, \kappa(m, p) \, \frac{(1-p)^{-m}}{p},
  \]
  and
\[
  \Var\left(T_{m,p}\right) =
  m^2 \, H_m^{(2)} - m \, H_m + \kappa(m, p)^2 \, \frac{(1-p)^{-2m}}{p^2} - \kappa(m, p) \, \frac{(1-p)^{-m}}{p} + \eta(m, p).
\]
\end{lemma}

\begin{proof}
  Since $T_{m,p}-T_{m,0}\ge0$ by construction, we have~\cite[(21.24)]{Billingsley1995} that for any $k\in\NN$ 
  \[
  \E \left[(T_{m, p} - T_{m, 0})^k \, e^{t(T(m,p)-T(m,0))}\right] = M_{T(m,p)-T(m,0)}^{(k)}(t), \qquad t<0.
  \]
  It then follows from monotone convergence and Lemma~\ref{thm:coupling} that
  \begin{equation}
    \label{eq:kth moment from I}
    \E \left[(T_{m, p} - T_{m, 0})^k\right] = \lim_{t\uparrow 0} I_{m, p}^{(k)}(t).
  \end{equation}
  For $t<0$, let
  \begin{equation*}
    F_{m,p}(t):=\int_{0}^1 p\,m\,(1-x)^{m\exp(-t)-m}\,(1-px)^{m-1} \diff x,
  \end{equation*}
  so that
  \begin{equation}\label{eq:I in terms of F}
  I_{m,p}(t)=(1-p)^m \left[1-F_{m,p}(t)\right]^{-1}.
\end{equation}
It is easily seen that for all $t<0$
\begin{equation}
  \label{eq:I' in terms of F'}
I_{m,p}'(t) = \frac{I_{m,p}(t)^2}{(1-p)^m} \, F_{m,p}'(t).
\end{equation}
  Since $\lim_{t\uparrow 0} I_{m,p}(t) = 1$ and
  \[
  \lim_{t\uparrow 0} F_{m,p}'(t) = -\int_{0}^1 p\,m^2 (1-px)^{m-1}\log(1-x) \, \diff x,
  \]
  we have
  \begin{align}
  \E \left[ T_{m, p} - T_{m, 0}\right]
  &= (1-p)^{-m} \,\lim_{t\uparrow 0} F_{m,p}'(t)\nonumber\\
  &=-(1-p)^{-m} \int_{0}^1 p\,m^2 (1-px)^{m-1}\log(1-x)\, \diff x.
    \label{eq:kappa before integration by parts}
  \end{align}
  Suppose, momentarily, that $m>1$. Observing that $L'(x)=-\log(1-x)$ for all $x<1$, we perform an integration by parts to obtain    
  \begin{align}
    \label{eq:diff mean}
    \E\left[T_{m,p} - T_{m, 0}\right]
    &= (1-p)^{-m}
      \Biggl\{p\,m^2\left(1-p\right)^{m-1}+\int_0^1 (p \, m)^2 \left(m-1\right) \left(1-p\,x\right)^{m-2} L(x) \, \diff x \Biggr\}
    \nonumber\\
    &= \kappa(m, p) \, \frac{(1-p)^{-m}}{p}.
  \end{align}
  Direct calculation shows that this expression also holds when $m=1$.
  
  Similarly, it follows from~\eqref{eq:I in terms of F} and~\eqref{eq:I' in terms of F'} that for all $t<0$
  \begin{equation}
    I_{m,p}''(t) = (1-p)^{-m}I_{m,p}(t)^2\,F_{m,p}''(t) +
    2\,\frac{[I_{m,p}'(t)]^2}{I_{m,p}(t)}.
    \label{eq:I''}
\end{equation}
  Consequently, \eqref{eq:kth moment from I}, \eqref{eq:diff mean} and~\eqref{eq:I''} imply
\begin{align}
  \Var(T_{m,p}-T_{m,0}) &= \lim_{t\uparrow 0} I_{m,p}''(t) - \left(\lim_{t\uparrow0}I_{m,p}'(t)\right)^2\nonumber\\
                        &= (1-p)^{-m}\lim_{t\uparrow 0} F_{m,p}''(t) + (1-p)^{-2m}\,\frac{\kappa(m,p)^2}{p^2}.
                          \label{eq:diff variance in terms of F}
\end{align}
But applying~\eqref{eq:kappa before integration by parts} and~\eqref{eq:diff mean} we have
\[
\lim_{t\uparrow0}F_{m,p}''(t) = -\frac{\kappa(m,p)}{p} + \int_0^1 p\,m^3(1-px)^{m-1}\log^2(1-x)\,\diff x.
\]
Suppose again, momentarily, that $m>1$. Observing that $N'(x)=\log^2(1-x)$ for $x<1$,
we perform an integration by parts to obtain
\begin{align}
  \lim_{t\uparrow0}F_{m,p}''(t)
  &= -\frac{\kappa(m,p)}{p} + 2\,p\,m^3(1-p)^{m-1} + \int_0^1\, p^2\,m^3(m-1)\,
  (1-px)^{m-2}\,N(x)\,\diff x \nonumber\\
  &= -\frac{\kappa(m,p)}{p} + (1-p)^{m}\,\eta(m,p).
 \label{eq:eta integration by parts}
\end{align}
Direct calculation shows that this expression also holds when $m=1$. Combining~\eqref{eq:eta integration by parts}
with~\eqref{eq:diff variance in terms of F} we then obtain
\begin{equation}
  \Var(T_{m,p}-T_{m,0}) = \kappa(m,p)^2\frac{(1-p)^{-2m}}{p^2} - \kappa(m,p)\frac{(1-p)^{-m}}{p} + \eta(m,p).
  \label{eq:diff variance}
\end{equation}
Finally, Theorem~\ref{thm:expectation} implies $\E T_{m,0} = m \, H_m$, and combining this
with~\eqref{eq:diff mean} yields the stated result for $\E T_{m,p}$. And similarly, Theorem~\ref{thm:variance} implies $\Var(T_{m,0})=m^2 H_m^{(2)} -m \,H_m$, and combining this with~\eqref{eq:diff variance} and the independence of $T_{m,0}$ and $T_{m,p}-T_{m,0}$ from Lemma~\ref{thm:coupling}
yields the stated result for $\Var(T_{m,p})$.
  \end{proof}

\begin{proof}[Proof of Theorem~\ref{thm:first moment}]
  Consider first the subcritical case, $p = o(1/m)$, and suppose $m$ is sufficiently large that $m\ge2$ and $p\,m\le 1/2$.
  Bernoulli's inequality then implies that
  \begin{equation}
    1 \leq (1-p)^{-m} \le \frac{1}{1-p\,m} \le 1+2\,p\,m,
    \label{eq:Bernoulli inequality}  
  \end{equation}
  and hence
  \[
    \frac{1}{4} \le (1-p)^{-2}\int_0^1 L(x)\,\diff x \le
    \int_0^1\frac{(1-p\,x)^{m-2}}{(1-p)^m}\, L(x)\,\diff x \le (1-p)^{-m}\int_0^1 L(x)\,\diff x \le \frac{1}{4}+ \frac{p\,m}{2}.
  \]
  Therefore
  \[
  \int_0^1 \frac{(1-p\,x)^{m-2}}{(1-p)^m}\, L(x)\,\diff x = \frac{1}{4}[1+O(pm)], \qquad m\to\infty.
  \]
  It then follows from~\eqref{eq:diff mean} that
  \begin{align*}
    \E[T_{m,p} - T_{m,0}] &= p\,m^2[1+O(p)] + \frac{p^2m^3}{4}[1+O(1/m)][1+O(p\,m)] \\
    &= p\,m^2 + \frac{p^2m^3}{4} + O(p^2m^2) + O(p^3m^4)
  \end{align*}
  which is the stated result.
  
  Next, suppose $p \sim c/m$  with $c>0$ a fixed constant.
  Then $(1-p)^m \to e^{-c}$ as $m \to \infty$, and so applying the dominated convergence theorem
  to \eqref{eq:diff mean}, we obtain
  \begin{equation}
    \label{eq:A}
    \begin{aligned}
    \lim_{m\to\infty}\frac{\E \left[ T_{m,p} - T_{m,0} \right]}{m}
    &=c + c^2 e^c \int_0^1 e^{-cx} \, L(x) \, \diff x\\
    &\eqqcolon A(c).
    \end{aligned}
  \end{equation}
  We now verify the stated properties of $A$. Firstly, since
  \begin{equation}
  \begin{aligned}
    \label{eq:A is increasing}
    A'(c)
    &= 1 + 2\,c\,e^c \int_0^1 e^{-cx} \, L(x)
    \, \diff x
    + c^2 \, e^c \int_0^1 \left(1-x\right)
    e^{-cx} \, L(x)\,\diff x\\
    &\geq 1
  \end{aligned}
  \end{equation}
  we see that $A$ is strictly increasing.
  Next, since $L(x)$ is nonnegative, we have
  \[
    c \leq A(c) \leq c + \frac{c^2e^c}{4},
  \]
  which shows $A(c) = c + O(c^2)$ as $c \downarrow 0$.
  And Laplace's method~\cite[p.~58]{wong1989asymptotics} implies that
  \begin{equation}
    \label{eq:critical Laplace}
    \int_0^1 e^{-cx}\,L(x) \, \diff x = \frac{1}{c^3}\left[1 + O\left(\frac{1}{c}\right)\right], \qquad c\to\infty.
  \end{equation}
  Hence, $A(c) = (e^c/c)\,(1 + O(c^{-1}))$ as $c \to \infty$, which establishes the stated properties of $A$.

  We now turn our attention to the case $p = \omega(1/m)$. From~\eqref{eq:diff mean},
  the key task is to determine the $m\to\infty$ asymptotics of 
  \[
  J_{m,p}\coloneqq \int_0^1 (1-p\,x)^{m-2} L(x) \diff x.
  \]
  Setting
  \[
    g_p(x)\coloneqq -\log(1-p\,x)/p =\sum_{\ell=1}^\infty \frac{p^{\ell-1}x^{\ell}}{\ell} 
  \]
  and $\lambda\coloneqq p\,(m-2)$ we have 
  \[
  J_{m,p} = \int_0^1 e^{-\lambda g_p(x)} L(x) \diff x.
  \]
  When $p$ depends on $m$, the $p$-dependence of $g_p$ somewhat complicates the application
  of Laplace's method, and it is convenient to therefore observe that 
  \[
  \int_0^1 e^{-\lambda g_1(x)} L(x) \diff x \le J_{m,p} \le \int_0^1 e^{-\lambda x} L(x) \diff x. 
  \]
  The asymptotics of the upper bound is given by~\eqref{eq:critical Laplace}, and a similar
  application of Laplace's method~\cite[p.~58]{wong1989asymptotics} to the lower bound shows that we likewise have
  \[
  \int_0^1 e^{-\lambda g_1(x)} L(x) \diff x = \frac{1}{\lambda^3}\left[1 + O\left(\frac{1}{\lambda}\right)\right], \qquad \lambda\to\infty.
  \]
  It therefore follows that
  \[
  J_{m,p} = \frac{1}{(p\,m)^3}\left[1 + O\left(\frac{1}{p\,m}\right)\right], \qquad m\to\infty.
  \]
  Since $(p\,m)^2(1-p)^{m-1}\le e^p (p\,m)^2 e^{-pm}$, it now follows that
  \begin{equation}
    \label{eq:supercritical kappa asymptotics}
  \kappa(m,p) = 1 + O\left(\frac{1}{p\,m}\right).
\end{equation}
Substituting~\eqref{eq:supercritical kappa asymptotics} into~\eqref{eq:diff mean} now shows that
\[
\E[T_{m,p}-T_{m,0}] = \frac{(1-p)^{-m}}{p}\left[1+O\left(\frac{1}{mp}\right)\right],
\]
as required. 

Finally, suppose that $p\in(0,1)$ is constant, independent of $m$, and consider
\begin{equation}
  \label{eq: Laplace for constant p mean}
J_{m,p} = \int_0^1 e^{2pg_p(x)} L(x)\,e^{-m pg_p(x)} \diff x.
\end{equation}
It is easily verified that, as $x\downarrow 0$, we have
  \begin{equation}
    \begin{aligned}
    \label{eq:const p Laplace}
      e^{2pg_p(x)}L(x) &= \frac{1}{2} \, x^2 + \frac{6\,p + 1}{6}\, x^3 + \frac{18\,p^2 + 4\,p + 1}{12} \, x^4 + O\left(x^5\right),\\
      p\,g_p(x) &= p\,x + \frac{p^2}{2}\, x^2 + \frac{p^3}{3}\, x^3 + O\left(x^4\right).
    \end{aligned}
  \end{equation}
Utilising these expansions, Laplace's method~\cite[p.~58]{wong1989asymptotics} then implies
\[
J_{m,p} = \frac{p^{-3}}{m^3} + \frac{p^{-4}}{m^4} + \frac{(p^{-3} - 2\,p^{-4} + 2\,p^{-5})}{m^5} + O\left(\frac{1}{m^6}\right).
\]
It then follows that
  \begin{align*}
p^3m^2(m-1)\,J_{m,p}
    &= (m-1)\left[\frac{1}{m} + \frac{p^{-1}}{m^2} + \frac{1 - 2\,p^{-1} + 2\,p^{-2}}{m^3} + O\left(\frac{1}{m^4}\right)\right]\\
    &= 1 + \frac{p^{-1}-1}{m} + \frac{1 - 3\,p^{-1} + 2\,p^{-2}}{m^2} + O\left(\frac{1}{m^3}\right),
  \end{align*}
  and therefore, since $(1-p)^{m-1}$ is exponentially small,
  \begin{equation}
    \kappa(m,p)=1 + \frac{p^{-1}-1}{m} + \frac{1 - 3\,p^{-1} + 2\,p^{-2}}{m^2} + O\left(\frac{1}{m^3}\right), \qquad m\to\infty.
    \label{eq:fixed p kappa asymptotics}
\end{equation}
Substituting~\eqref{eq:fixed p kappa asymptotics} into~\eqref{eq:diff mean} shows that
\[
\E[T_{m,p}-T_{m,0}] = \frac{(1-p)^{-m}}{p}\left[1 + \frac{p^{-1}-1}{m} + \frac{1 - 3\,p^{-1} + 2\,p^{-2}}{m^2} + O\left(\frac{1}{m^3}\right)\right], \qquad m\to\infty.
\]
We conclude by observing that since $m\,H_m\sim m\log(m)$ while $(1-p)^{-m}$ is exponentially large, we similarly have
\[
\E[T_{m,p}] = \frac{(1-p)^{-m}}{p}\left[1 + \frac{p^{-1}-1}{m} + \frac{1 - 3\,p^{-1} + 2\,p^{-2}}{m^2} + O\left(\frac{1}{m^3}\right)\right], \qquad m\to\infty.
\qedhere
\]
\end{proof}

Our final task is to prove Theorem~\ref{thm:second moment}. Since the asymptotics of $\kappa(m,p)$
was studied in the proof of Theorem~\ref{thm:first moment}, it follows from Lemma~\ref{thm:int first moment}
that it remains only to determine the asymptotics of $\eta(m,p)$.

\begin{proof}[Proof of Theorem~\ref{thm:second moment}]
  Our starting point is the identity
  \begin{equation}
    \label{eq:diff var in terms of diff mean}
      \Var\left(T_{m,p} - T_{m,0}\right)= \left(\E[T_{m,p}-T_{m,0}]\right)^2 - \E \left[T_{m,p} - T_{m,0}\right] + \eta(m, p),
    \end{equation}
  which is a direct consequence of~\eqref{eq:diff mean} and~\eqref{eq:diff variance}. 
  
  Consider first the subcritical case, $p = o(1/m)$, and suppose $m$ is sufficiently large that $m\ge2$ and $p\,m\le 1/2$.
  Then, arguing analogously to the subcritical case in the proof of Theorem~\ref{thm:first moment}, it follows
  from~\eqref{eq:Bernoulli inequality} that
  \[
  \int_0^1 \frac{(1-p\,x)^{m-2}}{(1-p)^m}\, N(x)\, \diff x = \frac{1}{4}[1+O(p\,m)], \qquad m\to\infty.
  \]
  Consequently, 
  \begin{equation}
      \label{eq:eta subcrit}
      \eta(m, p)= 2\,p\,m^3 + \frac{1}{4} \, p^2 \, m^4 + O\left(p^3\,m^5\right) + O\left(p^2\,m^3\right).
    \end{equation}
    Combining this with~\eqref{eq:diff var in terms of diff mean} and Theorem~\ref{thm:first moment} yields the stated result.

  Next, suppose $p \sim c/m$ with $c>0$ is a fixed constant.
  Then, applying dominated convergence, we have
  \begin{equation}
    \label{eq:B}
    \begin{aligned}
     \lim_{m\to\infty} \frac{\eta(m, p)}{m^2}
      &= 2\,c + c^2 e^c \int_0^1 e^{-cx} \, N(x) \, \diff x\\
      &\eqqcolon B(c)
    \end{aligned}
  \end{equation}
  Combining this with~\eqref{eq:diff var in terms of diff mean} and Theorem~\ref{thm:first moment} yields
  \[
  \Var(T_{m,p}-T_{m,0}) \sim \left[A(c)^2+B(c)\right] m^2, \qquad m\to\infty,
  \]
  as required. We now verify the stated properties of $B$. Firstly, since
  \begin{align*}
    B'(c) &= 2+ 2\,c\,e^c \int_{0}^1 e^{-cx} N(x)\, \diff x + c^2 e^c \int_0^1 e^{-cx} (1-x)\, N(x) \,\diff x\\
    &\ge 2
  \end{align*}
  we see that $B$ is strictly increasing. Next, since $N(x)$ is nonnegative, we have
  \[
  2\,c \le B(c) \le 2\,c + \frac{c^2 e^c}{4}
  \]
  which shows $B(c)=2c+O(c^2)$ as $c\downarrow 0$. And Laplace's method~\cite[p.~58]{wong1989asymptotics} implies that
  \begin{equation*}
    \int_0^1 e^{-cx}\,N(x) \, \diff x =\frac{2}{c^4}\left[1+O\left(\frac{1}{c}\right)\right], \qquad c\to\infty. 
  \end{equation*}
  Hence, $B(c) = (2\,e^c/c^2)\,(1 + O(c^{-1}))$ as $c \to \infty$,
  which establishes the stated properties of $B$.

  Finally, suppose $p = \omega(1/m)$. It is elementary to see from~\eqref{eq:eta def} that
  \[
  p^2 (1-p)^{2m}\, \eta(m,p) = O\left((p\,m)^4 e^{-pm}\right).
  \]
  Therefore, from Theorem~\ref{thm:first moment} we have
  \begin{equation}
    \label{eq:var diff supercritical}
    \Var(T_{m,p}-T_{m,0})
    = \frac{(1-p)^{-2m}}{p^2} \left( \left[p\,(1-p)^m \E(T_{m,p}-T_{m,0})\right]^2 + O\left((p\,m)^4 e^{-pm}\right)\right).
  \end{equation}
  Applying Theorem~\ref{thm:first moment} to~\eqref{eq:var diff supercritical} then yields
  \[
  \Var(T_{m,p}-T_{m,0})=\frac{(1-p)^{-2m}}{p^2}\left[1+O\left(\frac{1}{p\,m}\right)\right]
  \]
  as required.
  
  Finally, suppose in particular that $p \in (0, 1)$ is constant.
  It follows from Theorem~\ref{thm:first moment} that
  \[
  \left[p\,(1-p)^m \E(T_{m,p}-T_{m,0})\right]^2 =
  1+\frac{2p^{-1}-2}{m} + \frac{5p^{-2}-8p^{-1}+3}{m^2} + O\left(\frac{1}{m^3}\right).
  \]
  Substituting this into~\eqref{eq:var diff supercritical} then yields
  \[
  \Var(T_{m,p}-T_{m,0})
   = \frac{(1-p)^{-2m}}{p^2} \left[1+\frac{2p^{-1}-2}{m} + \frac{5p^{-2}-8p^{-1}+3}{m^2} + O\left(\frac{1}{m^3}\right)\right].
  \]
  We conclude by observing that since $m^2 H^{(2)}_m-m\,H_m = O(m^2)$ while $(1-p)^{-2m}$ is exponentially large, we similarly have
  \[
  \Var(T_{m,p})
  = \frac{(1-p)^{-2m}}{p^2} \left[1+\frac{2\,p^{-1}-2}{m} + \frac{5\,p^{-2}-8\,p^{-1}+3}{m^2} + O\left(\frac{1}{m^3}\right)\right].
  \qedhere
  \]
\end{proof}



\providecommand{\bysame}{\leavevmode\hbox to3em{\hrulefill}\thinspace}
\providecommand{\MR}{\relax\ifhmode\unskip\space\fi MR }
\providecommand{\MRhref}[2]{%
  \href{http://www.ams.org/mathscinet-getitem?mr=#1}{#2}
}
\providecommand{\href}[2]{#2}

\end{document}